\documentclass[reqno]{amsart}
\usepackage{amsmath,amsfonts,amssymb,color}
\newtheorem{teor}{Theorem}[section]
\newtheorem{lema}[teor]{Lemma}
\newtheorem{prop}[teor]{Proposition}

\theoremstyle{definition}

\newtheorem{eje}[teor]{Example}

\newtheorem{nota}[teor]{Remark}
\newtheorem{notas}[teor]{Remarks}

\numberwithin{equation}{section}
\newcommand{\R}{\mathbb{R}}

\newcommand{\A}{\mathbb{A}}
\newcommand{\B}{\mathbb{B}}

\newcommand{\dd}{\textsf{d}}

\newcommand{\wit}{\widetilde}

\newcommand{\lsm}{\left[\!\begin{smallmatrix}}
\newcommand{\rsm}{\end{smallmatrix}\!\right]}

\newcommand{\des}{\displaystyle}

\DeclareMathOperator{\Int}{Int} 
 
\DeclareMathOperator{\supp}{supp}

\begin{document}
\title[Linear-dissipative and purely dissipative parabolic PDEs]{Non-autonomous scalar linear-dissipative and purely dissipative parabolic PDEs over a compact base flow}
\author[R. Obaya]{Rafael Obaya}
\author[A.M. Sanz]{Ana M. Sanz}
\address[R. Obaya]{Departamento de Matem\'{a}tica
Aplicada, E. Ingenier\'{\i}as Industriales, Universidad de Valladolid,
47011 Valladolid, Spain, and member of IMUVA, Instituto de Investigaci\'{o}n en
Matem\'{a}ticas, Universidad de Valladolid, Spain.}
 \email{rafoba@wmatem.eis.uva.es}
\address[A.M. Sanz]{Departamento de Did\'{a}ctica de las Ciencias Experimentales, Sociales y de la Matem\'{a}tica,
Facultad de Educaci\'{o}n, Universidad de Valladolid, 34004 Palencia, Spain,
and member of IMUVA, Instituto de Investigaci\'{o}n en  Mate\-m\'{a}\-ti\-cas, Universidad de
Valladolid.} \email{anasan@wmatem.eis.uva.es}
\thanks{Both authors were partly supported by FEDER Ministerio de Econom\'{\i}a y Competitividad grants
MTM2015-66330-P and RTI2018-096523-B-I00 and by Universidad de Valladolid under project PIP-TCESC-2020}
\date{}
\begin{abstract}
In this paper a family of non-autonomous scalar parabolic PDEs over a general compact and connected flow is considered. The existence or not of a neighbourhood of  zero where the problems are linear has an influence on the methods used and on the dynamics of the induced skew-product semiflow. That is why two cases are distinguished:  linear-dissipative and purely dissipative problems. In both cases, the structure of the global and pullback attractors is studied using principal spectral theory. Besides, in the purely dissipative setting, a simple condition is given, involving both the underlying linear dynamics and some properties of the nonlinear term, to determine the nontrivial sections of the attractor.   \end{abstract}
\renewcommand{\subjclassname}{\textup{2010} Mathematics Subject Classification}
\maketitle
\section{Introduction}\label{sec-intro}\noindent
In this paper we investigate a family of time-dependent scalar linear-dissipative parabolic PDEs over a continuous flow on a compact metric space $(P,\sigma,\R)$, with Neumann, Robin or Dirichlet boundary conditions, of the form
\begin{equation*}
\left\{\begin{array}{l} \des\frac{\partial y}{\partial t}  =
 \Delta \, y+h(p{\cdot}t,x)\,y+g(p{\cdot}t,x,y)\,,\quad t>0\,,\;\,x\in U, \;\, \text{for each}\; p\in P,
  \\[.2cm]
By:=\bar\alpha(x)\,y+\kappa\,\des\frac{\partial y}{\partial n} =0\,,\quad  t>0\,,\;\,x\in \partial U,\,
\end{array}\right.
\end{equation*}
where $p{\cdot}t=\sigma_t(p)$, $h$ denotes the coefficient of the linear part of the problems, which is continuous, and $g$ is a smooth dissipative term with $g(p,x,0)=\frac{\partial g}{\partial y}(p,x,0)=0$   for every $p \in P $ and $x \in {\bar U}$. We assume that $g$ is dominant with respect to the linear term for $y$ large enough and that there is a constant $r_0\geq 0$ such that $g(p,x,y)=0$ if and only if $|y| \leq r_0$,  which determines the structure near zero. The long-term behaviour of the solutions is studied taking into account the internal dynamics of $P$. This dynamical formulation  is similar to the one in  Mierczy{\'n}ski and Shen~\cite{mish08} to investigate deterministic time-dependent parabolic equations, with applications to Kolmogorov models.
\par
We consider standard regularity assumptions which provide existence, uniqueness, existence in the large and continuous dependence of mild solutions with respect to initial data. Then, mild solutions of the abstract Cauchy problems (ACPs for short) associated to the linear-dissipative equations generate a global skew-product semiflow $\tau$ on $P \times X$, where $X=C(\bar U)$ for Neumann and Robin boundary conditions and $X=C_0(\bar U)$ in the Dirichlet case.
\par
Although the term {\em linear-dissipative\/}   is in general applied to the complete family of the previous parabolic equations for all the values of $r_0\geq 0$, in this paper and for the sake of distinction, we will refer to the {\em linear-dissipative\/} case provided that $r_0>0$,  which implies that there is a zone around $0$ where the problems are in fact linear and a zone away from $0$ where the dissipative term comes into play, whereas the term {\em purely dissipative\/} will be employed whenever $r_0=0$,   in order to emphasize the fact that the dissipative term is always active and the problems are strictly nonlinear in every neighbourhood of zero.
\par
The global dynamics of the semiflow $\tau$ is dissipative, which implies the existence of a global attractor $\A$.
Due to the monotonicity of $\tau$, $\A$ has lower and upper boundaries which can be identified with the graphs of two semicontinuous functions $a$ and $b$, respectively. Assuming for simplicity that $g(p,x,y)$ is odd with respect to $y$,  $a=-b$ and  the global attractor can be represented  as $\A=\cup_{p \in P} \{p\}\times A(p)$, with sections $A(p) \subseteq [-b(p),b(p)]$ for $p \in P$. It turns out that for each $p \in P$ the parametrized set $\{A(p{\cdot}t)\}_{t \in \R}$ is the pullback attractor of the process $S_p(\cdot,\cdot)$ generated  by the solutions of the non-autonomous parabolic equation obtained by the evaluation  of the coefficients along the trajectory of $p$. The parametric family of compact sets $\{A(p)\}_{p \in P}$ takes the name of cocycle (or pullback) attractor of the semiflow. Some nice references for attractors in non-autonomous dynamics are Carvalho et al.~\cite{calaro} and Kloeden and Rasmussen~\cite{klra}.
\par
Since the underlying linear dynamics is decisive in the description of the attractor, we also consider the global linear skew-product semiflow $\tau_L$ induced by the solutions of the linear part of the parabolic problems.  A crucial fact in this theory is that $\tau_L$ is strongly positive in $P \times X^\gamma$, where $X^\gamma=C({\bar U})$  for Neumann and Robin boundary conditions and $X^\gamma=X^\alpha$, a fractional power space, in the case of Dirichlet boundary conditions.  In consequence, $\tau_L$ admits a continuous separation $X^\gamma=X_1(p)\oplus X_2(p)$ for  $p \in P$, in the terms stated in Pol\'{a}\v{c}ik and Tere\v{s}\v{c}\'{a}k~\cite{pote} and Shen and Yi~\cite{shyi}. The restriction of $\tau_L$ to the principal bundle $\cup_{p \in P}\{p\} \times X_1(p)$ defines a continuous one-dimensional (1-dim for short) linear flow determined by a 1-dim linear cocycle $c(t,p)$ for $t \in \R$ and $p \in P$, and its continuous spectrum  is called the {\em principal spectrum\/}. The 1-dim cocycle $c(t,p)$  determines  the asymptotic behaviour of the positive solutions of $\tau_L$ and plays an important role in the description of both topological and ergodic properties of $\tau$ on the global attractor.
\par
We investigate the structure and internal dynamics of the global and cocycle attractors of $\tau$ in terms of the principal spectrum of $\tau_L$.
These questions were initiated in Cardoso et al.~\cite{cardoso}, Caraballo et al.~\cite{calaobsa} and Langa et al.~\cite{laos} when the flow in $P$ is minimal and uniquely ergodic, and the principal spectrum is thus a singleton. Although purely dissipative problems are also considered, these references extensively cover the linear-dissipative case when the flow on $P$ is aperiodic, showing ingredients of high complexity or chaotic dynamics that cannot occur when the equations are autonomous or periodic.
In the present paper,  for both the linear-dissipative and the purely dissipative settings we give a version of the theory valid for a general continuous flow on a compact base $P$. In particular we describe the dynamical complexity when several ergodic measures or minimal subsets coexist in $P$.  Besides, for a very general class of purely dissipative problems,  we determine integrability conditions on the cocycle $c(t,p)$ on $(-\infty,0]$ for a fixed $p \in P$ which imply that the section of the attractor $A(p)$ is nontrivial. Similar conditions are necessary for $A(p)\not=\{0\}$ with Neumann or Robin boundary conditions, under some spectral restrictions. The conclusions in the purely dissipative case are compatible with the presence of ergodic measures with null Lyapunov exponent and are new also when the flow on the base $P$ is minimal or even almost periodic.
\par
The paper is organized as follows. Section~\ref{sec-preli} compiles some basic notions in the field of monotone non-autonomous dynamical systems. In Section~\ref{sect-the problem} the family of  PDEs problems under study is detailed,  and known results on the existence of a global attractor for the induced skew-product semiflows are reviewed. Also, some general results are given for the boundary map $b$ and its null set. Some aspects of the linear semiflow $\tau_L$ are treated in Section~\ref{sec-lineal}. In particular, the principal spectrum and the 1-dim linear cocycle $c(t,p)$ are introduced, which have leading roles in the paper. In order to have that the principal spectrum is a possibly degenerate compact interval $\Sigma_{\text{pr}}=[\alpha_P,\lambda_P]$ on the real line, in the rest of the paper we assume that the compact base $P$ is also connected.
\par
In Section~\ref{sect-linear dissipative} we consider linear-dissipative problems, that is, $r_0>0$. One advantage in this case is that $b(p)\gg 0$ is characterized by the boundedness of  $c(t,p)$ for negative times. A description of the map $b$ is given depending on the location of the principal spectrum $\Sigma_{\text{pr}}$ on the real line. The only chance for a trivial attractor happens for $\lambda_P<0$.   If $\Sigma_{\text{pr}}=[\alpha_P,0]$ with $\alpha_P \leq 0$, there is some $p \in P$ with $b(p)\gg 0$. The set of these points can even have measure $1$ for some ergodic ergodic measure with null Lyapunov exponent, giving rise to a big subset of the phase space with chaotic dynamics. If $\Sigma_{\text{pr}}=[0,\lambda_P]$ with $\lambda_P \geq 0$,  the previous chaotic dynamics can also appear
 and in addition every minimal set $M$ in $P$ has some point $p\in M$ with $b(p)\gg 0$. Besides, if $\lambda_P>0$, $b$ is almost always strongly positive with respect to an ergodic measure and, if $\alpha_P>0$, $b$ is uniformly strongly positive.
Last but not least, when the nonlinear term $g$ is strictly sublinear in the zone $y>r_0$, a more detailed dynamical description  is obtained and, with $\alpha_P>0$, $b$ is continuous.
\par
The purely dissipative case, with $r_0=0$, is treated in Section~\ref{sect-purely dissipative}. The internal dynamics of the attractor is described once more according to the location of $\Sigma_{\text{pr}}$. Also now further information is obtained under a strictly sublinear condition: we note that some of the results had been proved by Mierczy{\'n}ski and Shen~\cite{mish04, mish08} for Kolmogorov models. But in fact the most important result in the paper is Theorem~\ref{teor-atractor no trivial}, determining a sufficient (and many times necessary) integrability condition on the 1-dim cocycle $c(t,p)$ for $t\leq 0$ to guarantee that $b(p)\gg 0$, eventually applicable to a very general class of problems in Theorem~\ref{teor-coro}. One can find in the literature  classical examples  by Poincar\'{e}~\cite{poin}, Conley and Miller~\cite{comi}, Zhikov and Levitan~\cite{zhle} or Johnson and Moser~\cite{jomo} of almost periodic maps satisfying the cited integrability condition.  Finally, Section~\ref{sec-examples} offers a collection of simple examples to illustrate different dynamical behaviours in the attractor.
%%%%%%%%%%%%%%%%%%%%%%%%%%%%%%%%%%%%%%%%%%%%%%%%%%%%%%%%%%%%%%%%%%%%%%%%%%%%%%%%%%%%%%%%%%
\section{Basic notions}\label{sec-preli}\noindent
In this section we include some preliminaries about
topological dynamics for non-autonomous dynamical systems.
\par
Let $(P,d)$ be a compact
metric space. A real {\em continuous flow\/} $(P,\sigma,\R)$ is
defined by a continuous map $\sigma: \R\times P \to  P,\;
(t,p)\mapsto \sigma(t,p)=\sigma_t(p)$ satisfying
\begin{enumerate}
\renewcommand{\labelenumi}{(\roman{enumi})}
\item $\sigma_0=\text{Id},$
\item $\sigma_{t+s}=\sigma_t\circ\sigma_s$ for each $s$, $t\in\R$\,.
\end{enumerate}
The set $\{ \sigma_t(p) \mid t\in\R\}$ is called the {\em orbit\/}
of the point $p$. A point $p$ is called {\em topologically transitive\/} if its orbit is dense in $P$. We say that a subset $P_1\subset P$ is  {\em invariant\/} if $\sigma_t(P_1)=P_1$ for every $t\in\R$.
The flow $(P,\sigma,\R)$ is called {\em minimal\/} if it does not contain properly any other
compact invariant set, or equivalently,  if every
orbit is dense.
The flow is {\em distal\/} if the orbits of  any two distinct
points $p_1,\,p_2\in P$  keep at a positive distance,
that is, $\inf_{t\in \R}d(\sigma(t,p_1),\sigma(t,p_2))>0$; and it
is {\em almost periodic\/} if the family of maps $\{\sigma_t\}_{t\in \R}:P\to P$ is uniformly equicontinuous. An almost periodic flow is always distal.
\par
A finite regular measure defined on the Borel sets of $P$ is called
a Borel measure on $P$. Given $\mu$ a normalized Borel measure on
$P$, it is  {\em invariant\/} for the flow $\sigma$ if $\mu(\sigma_t(P_1))=\mu(P_1)$ for every Borel subset
$P_1\subset P$ and every $t\in \R$. It is {\em ergodic\/}  if, in
addition, $\mu(P_1)=0$ or $\mu(P_1)=1$ for every
invariant Borel subset $P_1\subset P$. An invariant Borel set $P_1\subseteq P$ is {\em of complete measure\/} if $\mu(P_1)=1$ for every ergodic measure $\mu$. If that is the case, then $\mu(P_1)=1$ for every invariant measure.
$(P,\sigma,\R)$ is {\em uniquely ergodic\/} if it has a
unique normalized invariant measure, which is then necessarily
ergodic. A minimal and almost periodic flow $(P,\sigma,\R)$ is uniquely ergodic. For simplicity, we will denote the flow by $\sigma_t(p)=p{\cdot}t$ for $p\in P$, $t\in\R$.
\par
A standard method to, roughly speaking, get rid of the time variation in a non-autonomous equation and build a non-autonomous dynamical system, is the so-called {\em hull\/} construction. More precisely, a function $f\in C(\R\times\R^m)$ is said to be {\em admissible\/} if for any
compact set $K\subset \R^m$, $f$ is bounded and uniformly continuous
on $\R\times K$. Provided that $f$ is admissible, its {\em hull\/} $P$ is the closure for the compact-open topology of the set of $t$-translates of $f$, $\{ f_t \mid t\in\R\}$ with $f_t(s,x)=f(t+s,x)$
for $s\in \R$ and $x\in\R^m$. The translation map $\R\times P\to P$,
$(t,p)\mapsto p{\cdot}t$ given by $p{\cdot}t(s,x)= p(s+t,x)$, $s\in \R$ and $x\in\R^m$, defines a
continuous flow on the compact metric space $P$. This flow is minimal as far as the map $f$ has certain recurrent behaviour in time, such as periodicity, almost periodicity, or other weaker properties of recurrence. If the map $f(t,x)$ is uniformly almost periodic (that is, it is admissible and almost periodic in $t$ for any fixed $x$), then the flow on the hull is minimal and almost periodic, and thus uniquely ergodic. It is aperiodic whenever $f$ is not time-periodic. %This is how an almost periodic equation is brought into the abstract context of this paper.
\par
Let $\R_+=\{t\in\R\,|\,t\geq 0\}$. Given a continuous compact flow $(P,\sigma,\R)$ and a
complete metric space $(X,\dd)$, a continuous {\em skew-product semiflow\/} $(P\times
X,\tau,\,\R_+)$ on the product space $P\times X$ is determined by a continuous map
\begin{equation*}
 \begin{array}{cccl}
 \tau \colon  &\R_+\times P\times X& \longrightarrow & P\times X \\
& (t,p,x) & \mapsto &(p{\cdot}t,u(t,p,x))
\end{array}
\end{equation*}
 which preserves the flow on $P$, called the {\em base flow\/}.
 The semiflow property means:
\begin{enumerate}
\renewcommand{\labelenumi}{(\roman{enumi})}
\item $\tau_0=\text{Id},$
\item $\tau_{t+s}=\tau_t \circ \tau_s$ for  $t,s\geq 0,$
\end{enumerate}
where again $\tau_t(p,x)=\tau(t,p,x)$ for each $(p,x) \in P\times X$ and $t\in \R_+$.
This leads to the so-called (nonlinear) semicocycle property,
\begin{equation*}
 u(t+s,p,x)=u(t,p{\cdot}s,u(s,p,x))\quad\mbox{for $t,s\ge 0$ and $(p,x)\in P\times X$}.
\end{equation*}
\par
The set $\{ \tau(t,p,x)\mid t\geq 0\}$ is the {\em semiorbit\/} of
the point $(p,x)$. A subset  $K$ of $P\times X$ is {\em positively
invariant\/} if $\tau_t(K)\subseteq K$
for all $t\geq 0$ and it is $\tau$-{\em invariant\/} if $\tau_t(K)= K$
for all $t\geq 0$.  A compact $\tau$-invariant set $K$ for the
semiflow  is {\em minimal\/} if it does not contain any nonempty
compact $\tau$-invariant set  other than itself.
\par
A compact $\tau$-invariant set $K\subset P\times X$ is called a {\em pinched\/} set if there exists a residual set $P_0\subsetneq P$ such that for every $p\in P_0$ there is a unique element in $K$ with $p$ in the first component, whereas there are more than one if $p\notin P_0$.
\par
The reader can find in  Ellis~\cite{elli}, Sacker and
Sell~\cite{sase94}, Shen and Yi~\cite{shyi} and references therein, a
more in-depth survey on topological dynamics.
\par
We now state the definitions of global and cocycle (or pullback) attractors for skew-product semiflows. The books by    Carvalho et al.~\cite{calaro} and Kloeden and Rasmussen~\cite{klra}  are  good references for this topic.
We say that the skew-product semiflow $\tau$ has a {\em global attractor\/} if there exists an invariant compact set attracting bounded sets forwards in time, i.e., there is a compact set $\A\subset P\times X$ with $\tau_t(\A) = \A$ for any $t\geq 0$ and $\lim_{t\to\infty} {\rm dist}_{P\times X}(\tau_t(\B),\A)=0$ for any bounded set $\B\subset P\times X$, for the Hausdorff semidistance ${\rm dist}$ on $P\times X$.
\par
Besides,  provided that $P$ is compact, if $\A$ is the global attractor, the non-autonomous set $\{A(p)\}_{p\in P}$, with $A(p)=\{z\in X\mid (p,z)\in \A\}$ for each $p\in P$, is a {\em cocycle attractor\/} (or a {\em pullback attractor\/}). This means that $\{A(p)\}_{p\in P}$ is compact, invariant and it pullback attracts all bounded subsets $B\subset X$, that is, $\lim_{t\to\infty} {\rm dist}_X(u(t,p{\cdot}(-t),B),A(p))=0$ for any $p\in P$ (see~\cite{klra}).
\par
To finish, we include some basic notions on monotone skew-product semiflows. When the state space $X$ is a strongly ordered Banach space, that is, there is a closed convex cone of nonnegative vectors $X_+$ with a nonempty interior, then, a (partial) {\em strong order relation\/} on $X$ is
defined by
\begin{equation}\label{order}
\begin{split}
 x\le y \quad &\Longleftrightarrow \quad y-x\in X_+\,;\\
 x< y  \quad &\Longleftrightarrow \quad y-x\in X_+\;\text{ and }\;x\ne y\,;
\\  x\ll y \quad &\Longleftrightarrow \quad y-x\in \Int X_+\,.\qquad\quad\quad~
\end{split}
\end{equation}
In this situation, the skew-product semiflow $\tau$
is {\em monotone\/} if $u(t,p,x)\le u(t,p,y)$ for $t\ge 0$, $p\in P$ and  $x,y\in X$ with $x\le y$; and it is {\em strongly monotone\/} if besides, $u(t,p,x)\ll u(t,p,y)$ for $t>0$ and $p\in P$,  provided that $x< y$.
\par
A Borel map $a:P\to X$ such that $u(t,p,a(p))$ exists for any $t\geq 0$ is said to be
\begin{itemize}
\item[(a)] an {\em equilibrium\/} if $a(p{\cdot}t)=u(t,p,a(p))$ for any $p\in P$ and $t\geq 0$;
\item[(b)] a {\em sub-equilibrium\/} if $a(p{\cdot}t)\leq u(t,p,a(p))$ for any $p\in P$ and $t\geq 0$;
\item[(c)] a {\em super-equilibrium\/} if $a(p{\cdot}t)\geq u(t,p,a(p))$ for any $p\in P$ and $t\geq 0$.
\end{itemize}
The study of semicontinuity properties of these maps and other related issues can be found in Novo et al.~\cite{nono2}.
%%%%%%%%%%%%%%%%%%%%%%%%%%%%%%%%%%%%%%%%%%%%%%%%%%%%%%%%%%%%%%%%%%%%%%%%%%%%%%%%%%%%%%%%%%%%%%%%%%%%%%%%%%%%%%%
%%%%%%%%%%%%%%%%%%%%%%%%%%%%%%%%%%%%%%%%%%%%%%%%%%%%%%%%%%%%%%%%%%%%%%%%%%%%%%%%%%%%%%%%%%%%%%%%%%%%%%%%%%%%%%%%
\section{The setting of the problem and some general results}\label{sect-the problem}\noindent
This section contains the description of the parabolic problems under study. Some known results for them are reviewed and  some new general results are given.
\par
In this paper we consider a family of time-dependent scalar linear-dissipative parabolic PDEs over a continuous flow on a compact metric space $(P,\sigma,\R)$, with Neumann, Robin or Dirichlet boundary conditions, given for each $p\in P$ by
\begin{equation}\label{pdefamilynl}
\left\{\begin{array}{l} \des\frac{\partial y}{\partial t}  =
 \Delta \, y+h(p{\cdot}t,x)\,y+g(p{\cdot}t,x,y)\,,\quad t>0\,,\;\,x\in U,
  \\[.2cm]
By:=\bar\alpha(x)\,y+\kappa\,\des\frac{\partial y}{\partial n} =0\,,\quad  t>0\,,\;\,x\in \partial U,\,
\end{array}\right.
\end{equation}
where  $p{\cdot}t$ denotes the flow on $P$; $U$, the spatial domain, is a bounded, open and
connected  subset of $\R^m$ ($m\geq 1$) with a sufficiently smooth boundary
$\partial U$; $\Delta$ is the Laplacian operator on $\R^m$; the linear coefficient $h:P\times \bar U\to \R$ is continuous
and the nonlinear term $g:P\times \bar U\times \R\to \R$ is continuous and of class $C^1$ with respect to  $y$ and satisfies the following conditions which in particular render the equations dissipative:
\begin{itemize}
  \item[(c1)] $g(p,x,0)=\des\frac{\partial g}{\partial y}(p,x,0)=0$ for any $p\in  P$ and $x\in \bar U$;
  \item[(c2)] $y\,g(p,x,y)\leq 0$ for any $p\in  P$, $x\in \bar U$ and $y\in \R$;
  \item[(c3)] $\des\lim_{|y|\to\infty}\frac{g(p,x,y)}{y}=-\infty$ uniformly on $P\times\bar U$;
  \item[(c4)] $g(p,x,-y)=-g(p,x,y)$ for any $p\in  P$, $x\in \bar U$ and $y\in \R$;
  \item[(c5)] there exists an $r_0\geq 0$ such that $g(p,x,y)=0$ if and only if $|y|\leq r_0$.
\end{itemize}
\par
The problem has Dirichlet boundary conditions if $\kappa=0$ and $\bar\alpha(x)\equiv 1$; Neumann boundary conditions if $\kappa=1$ and  $\bar\alpha(x) \equiv 0$; and Robin boundary conditions if  $\kappa=1$ and $\bar\alpha:\partial U\to \R$ is a nonnegative  sufficiently regular map. As usual,
$\partial/\partial n$ denotes the
outward normal
derivative at the  boundary.
As it has been indicated in the introduction, we make a distinction between what we call linear-dissipative problems ($r_0>0$) and purely dissipative problems ($r_0=0$) in terms of (c5).
\par
Even if no further mention is made, the former regularity conditions on $h$ and $g$ are always presumed. Nonetheless, sometimes we will assume additional regularity conditions on $g$ with respect to the $y$ variable. Namely, given a $\delta>0$, an integer $n\geq 1$ and a $\beta\in (0,1^-]$ we write $C^{0,0,n+\beta}(P\times \bar U\times [0,\delta])$ for the space of real continuous maps $g$ on $P\times \bar U\times [0,\delta]$ which are $n$ times continuously differentiable with respect to $y$ and the $n^{\rm{th}}$-order partial derivative satisfies that:
\[
\sup\Big\{ \frac{|\frac{\partial^n g}{\partial y^n}(p,x,y_2)-\frac{\partial^n g}{\partial y^n}(p,x,y_1)|}{(y_2-y_1)^\beta} \,\Big|\; p\in P,\,x\in\bar U,\, 0\leq y_1<y_2\leq \delta\Big\}<\infty\,.
\]
Note that $\frac{\partial^n g}{\partial y^n}$  is $\beta$-H\"{o}lder continuous uniformly for $p\in P$ and $x\in\bar U$ if $\beta\in (0,1)$, whereas it is Lipschitz continuous if $\beta=1^-$. The superscript in $1^-$ is just used to distinguish the Lipschitz condition from a $C^1$ character.
\par
Also, on some occasions we will be able to give more complete information by adding a strict sublinearity condition on $g$:
\begin{itemize}
\item[(c6)] $g(p,x,\lambda\,y)< \lambda\,g(p,x,y)$  for $p\in P$,
$x\in\bar U$, $y>r_0$ and $\lambda >1$\,.
\end{itemize}
Note that the sublinear case includes the non-autonomous versions of some classical problems, such as the Chafee-Infante equation (see Chafee and Infante~\cite{chin} and Carvalho et al.~\cite{calaro12}), which in principle is purely dissipative, but we can allow $r_0>0$ to get new linear-dissipative versions of this classical problem. Also some Kolmogorov and Fisher models (see Shen and Yi~\cite{shyi2}) provide a field of application of our results, both in the linear-dissipative and purely dissipative settings.
\par
Assuming only conditions (c1), (c2) and (c3) on the nonlinear term $g$ in the parabolic PDEs~\eqref{pdefamilynl} with Neumann or Robin boundary conditions, plus a minimality condition on the base flow $P$, Cardoso et al.~\cite{cardoso} prove  the existence of a global attractor $\A$ for the induced skew-product semiflow
\begin{equation*}%\label{tau}
\begin{array}{cccl}
 \tau: & \R_+\times P\times X& \longrightarrow & \hspace{0.3cm}P\times X\\
 & (t,p,z) & \mapsto
 &(p{\cdot}t,u(t,p,z))
\end{array}
\end{equation*}
defined by the mild solutions of the associated ACPs in $X=C(\bar U)$ (see also~\cite{calaobsa} for all the details). The semiflow $\tau$ is globally defined because of the boundedness of the solutions, it is strongly monotone, and the section semiflow $\tau_t$ is compact for every $t>0$. All these results remain valid without the minimality condition on $P$, that is, just assuming that  $P$ is compact.
\par
The case of Dirichlet boundary conditions has recently been treated in Langa et al.~\cite{laos}. The reader is referred to this paper for all the details. The appropriate state space to build the skew-product semiflow $\tau$ induced by mild solutions is now $X=C_0(\bar U)$, the set of continuous maps on $\bar U$ vanishing  on the boundary $\partial U$. Once more, under assumptions (c1), (c2) and (c3) there exists a global attractor $\A$ for $\tau$. This is a consequence of Proposition~3.4 in~\cite{laos}, whose proof works under no further conditions on $P$ rather than compactness.
\par
An additional state space is introduced in the Dirichlet case in order to deal with strong monotonicity issues, which is  $X^\alpha$,  a fractional power space associated with the realization of the Laplacian in an $L^p(U)$ space, and in particular satisfies $X^\alpha\hookrightarrow C^1(\bar U)\cap C_0(\bar U)$ under some restrictions, namely, $m<p<\infty$ and $\alpha\in (1/2+m/(2p),1)$. The (partial) strong order in the Banach space $X^\alpha$ is defined as in~\eqref{order} in  association with the cone of positive maps $X^\alpha_+=\{z\in X^\alpha\,\big|\; z(x)\geq 0\;\text{for} \;
x\in\bar U \}$, which has a nonempty interior: %(recall that $X^\alpha \hookrightarrow C^1(\bar U)$)
\begin{equation*}
 \Int X^\alpha_+ =\Big\{z\in X^\alpha_+\,\big|\; z(x)> 0\;\text{for} \;
x\in U \;\text{and}\; \frac{\partial z}{\partial n}(x)< 0 \;\text{for} \;
x\in\partial U   \Big\}\,.
\end{equation*}
As it is, $u(t,p,z)\in X^\alpha$ for $t>0$, $p\in P$ and $z\in X$, the skew-product semiflow $\tau$ is strongly monotone and for each $t>0$ the section semiflow $\tau_t:P\times X\to P\times X^\alpha$ is compact. The global attractor $\A\subset P\times X^\alpha$ and the topologies of $P\times X$ and $P\times X^\alpha$ coincide over $\A$.
\begin{nota}
In order to unify the writing for any kind of boundary conditions, we will consider the Banach spaces $(X,\|\,{\cdot}\,\|)$ and $(X^\gamma,\|\,{\cdot}\,\|_\gamma)$:
\par
- With Neumann or Robin boundary conditions, $X=X^\gamma=C(\bar U)$, always with the sup-norm.\par
- With Dirichlet boundary conditions, $X=C_0(\bar U)$ with the sup-norm, whereas $X^\gamma=X^\alpha$ with the standard $\alpha-$norm $\|\,{\cdot}\,\|_\alpha$.
\end{nota}
One crucial dynamical element is the upper boundary map of the global attractor, that is, the map $b(p):=\sup A(p)$, for the sections $A(p)$ of the attractor $\A$ defined in the natural way, $A(p):=\{z\in X^\gamma\mid (p,z)\in \A\}$ for each $p\in P$. Note that with condition (c4), the lower boundary map $a(p):=\inf A(p)$, $p\in P$ satisfies  $a(p)= -b(p)$, and $\A\subseteq \bigcup_{p\in P} \{p\}\times [-b(p),b(p)]$. That is why the study is centered on $b$. With any kind of boundary conditions, the map $b:P\to X^\gamma$ defines a semicontinuous equilibrium for $\tau$, which is thus continuous over the points in an invariant and residual set of $P$. For convenience, we explain this fact further. The main ideas involved can be found in the work by Chueshov~\cite{chue} in a random setting, and in Novo et al.~\cite{nono2} in a deterministic framework.
Let $e_0\gg 0$ be the eigenfunction associated to the first eigenvalue of the boundary value problem (BVP for short)
\begin{equation}\label{bvp}
\left\{\begin{array}{l}
 \Delta \, u +\lambda\,u = 0\,,\quad x\in U,\\
Bu=0\,,\quad x\in \partial U,
\end{array}\right.
\end{equation}
with $\|e_0\|=1$.
In the case of Neumann or Robin boundary conditions, $b(p)$ is obtained as the decreasing limit of a family of continuous super-equilibria (see Proposition~3 in Cardoso et al.~\cite{cardoso}):
\begin{equation}\label{b(p)}
b(p)=\lim_{t\to\infty} u(t,p{\cdot}(-t),r e_0)\quad\text{for}\;\,p\in P,
\end{equation}
where $r>0$ is big enough; whereas with Dirichlet boundary conditions the construction is a bit more delicate and $b$ is obtained as the decreasing limit of a family of continuous super-equilibria for an associated discrete semiflow:
\begin{equation}\label{b(p)-dir}
b(p):=\lim_{n\to\infty} u(nt_0,p{\cdot}(-nt_0),re_0)\quad\text{for}\;\,p\in P,
\end{equation}
for $r>0$ and $t_0>0$ both big enough. This is done in Proposition~3.5 in Langa et al.~\cite{laos}. In fact, it can be derived from its proof that the limit in~\eqref{b(p)} exists, so that formula~\eqref{b(p)} is also applicable with Dirichlet boundary conditions.
\par
In order to study the upper boundary map $b$, the semiflow can be restricted to the positive cone. Also, in models of applied sciences often only nonnegative solutions  make sense. For these reasons, we state the following result.  Here, the strong monotonicity of $\tau$ is crucial.
\begin{prop}\label{prop-atractor cono positivo}
Assume conditions $(c1)$-$(c3)$ on $g$. The restriction of the semiflow $\tau$ to the  positive cone, that is, to $P\times X_+$, admits a global attractor $\A_+\subseteq \A$ which is contained in $(P\times\{0\}) \cup (P\times \Int X^\gamma_+)$. Besides, $\A$ and $\A_+$ share the upper boundary map $b$.
\end{prop}
\begin{proof}
Take the compact absorbing set $C_1\subset P\times X$ for $\tau$ given in Proposition~2 in~\cite{cardoso} or in Proposition~3.4 in~\cite{laos}, respectively for Neumann or Robin, and Dirichlet boundary conditions. It is obvious that $C_1\cap (P\times X_+)$ is a compact absorbing set for the restriction of the semiflow $\tau$ to the closed and positively invariant set   $P\times X_+$. Therefore, there exists a global attractor for the restricted semiflow, call it $\A_+\subset P\times X_+$. Clearly $\A_+\subseteq \A$ and the trivial null orbit lies inside $\A_+$. By the strong monotonicity, if $z>0$, then $u(t,p,z)\gg 0$ for any $t>0$ and $p\in P$. Now, take a pair $(p,z)\in \A_+$ with $z\not=0$. Since we have full orbits inside $\A_+$, we can write $z=u(t,p{\cdot}(-t),u(-t,p,z))$ for $(p{\cdot}(-t),u(-t,p,z))\in \A_+$ for any $t>0$. It must be $u(-t,p,z)\not=0$ since otherwise $z=0$, so that $u(-t,p,z)>0$ and then $z\gg 0$, as we wanted to prove. The fact that $\A$ and $\A_+$ share the upper boundary follows from formula~\eqref{b(p)} for $b(p)$. The proof is finished.
\end{proof}
As a consequence, we get some basic properties for the upper boundary of $\A$.
\begin{prop}\label{prop-0 o positivo}
Assume conditions $(c1)$-$(c3)$ on $g$. Then,  for any $p\in P$, either $b(p)=0$ or $b(p)\gg 0$. Besides, the sets $P_0:=\{p\in P\mid b(p)=0\}$
and $P_{+}:=\{p\in P\mid b(p)\gg 0\}$ are invariant and, if $R\subseteq P$ is the invariant and residual set of continuity points for $b$, then $P_0\subseteq R$.
\end{prop}
\begin{proof}
By Proposition \ref{prop-atractor cono positivo}, in particular $b(p)=0$ or $b(p)\gg 0$ for $p\in P$.
Now, if $b(p)=0$ for some $p\in P$, for $t>0$, $b(p{\cdot}t)=u(t,p,0)=0$;  and also $b(p{\cdot}(-t))=0$, for if not, we would get a contradiction moving forwards.  Thus, the set $P_0=\{p\in P\mid b(p)=0\}$ is invariant and so is its complementary set $P_+$.
\par
Finally, if $p_0$ is such that $b(p_0)=0$, then $p_0$ is a point of continuity for $b$ due to the semicontinuity of $b$, which in this case means that $\{(p,z)\mid z\leq b(p)\}$ is a closed set, and the fact that $b$ is always nonnegative.
\end{proof}
The size of the sets $P_0$ and $P_+$ will be studied in Sections~\ref{sect-linear dissipative} and~\ref{sect-purely dissipative}, at least from an ergodic point of view, depending on the associated principal spectrum. Maintaining the previous notation, we state the next result.
\begin{prop}
Assume conditions $(c1)$-$(c3)$ on $g$. Then:
\begin{itemize}
\item[(i)] $P$ decomposes as the disjoint union $P=V\cup F$ of an open invariant set $V\subseteq P_+$ and a compact invariant set $F$ (both possibly empty) in such a way that, if $F\not=\emptyset$, then $P_0=R\cap F$ is the residual subset on $F$ of continuity points for the restriction map $b_{|F}$.
\item[(ii)] $V$ is not void if and only if $R\cap P_+$ is not void.
\item[(iii)] If $p_0\in P$ is topologically transitive and $b(p_0)=0$, then $P_0=R$ and $P=F$.
\item[(iv)] If $P$ is minimal, then either $V=\emptyset$ or $F=\emptyset$.
\item[(v)] If there is an ergodic measure $\mu$ with support $\supp(\mu)=P$ and $V\not=\emptyset$, then $\mu(V)=1$ and $V$ is dense.
\end{itemize}
\end{prop}
\begin{proof}
(i) Let us take the open set $V:=\Int P_{+}$ and the compact set $F:=\overline{P_0}$. If there exists some $p\in \Int P_{+}$, then there is an open neighbourhood of $p$ contained in $P_+$, so that it cannot be $p\in \overline{P_0}$. On the other hand, if $p\notin \Int P_{+}$,  in any open neighbourhood of $p$ we can find a point which is not in $P_+$ and is thus in $P_{0}$, so that $p\in \overline{P_0}$. From these facts, we conclude that $V\cap F=\emptyset$ and $P=V\cup F$. Since $P_0$ is known to be invariant, its closure $F$ is invariant too, and so is its complementary set $V$. Finally, if $F\not=\emptyset$, it is obvious that $R\cap F$ is the residual subset on $F$ of continuity points for the restriction map $b_{|F}$ and $P_0\subseteq R\cap F$ by Proposition~\ref{prop-0 o positivo}. Conversely, if $p\in R\cap F$, it is immediate to check that $b(p)=0$ and $p\in P_0$.
\par
(ii) If $V=\Int P_{+} \not=\emptyset$, there is a ball in $P_{+}$, and since the residual set $R$ is dense, it must be $R\cap P_+\not=\emptyset$. Conversely,  if $R\cap P_+\not=\emptyset$, by continuity there is a ball inside $P_+$, so that  $V\not=\emptyset$.
\par
(iii) First of all, $b(p_0{\cdot}t)=0$ for any $t\in \R$ and $P_0\subseteq R$, by Proposition~\ref{prop-0 o positivo}.  Argue by contradiction and assume that $P_0\subsetneq R$. Then, $R\cap P_+\not=\emptyset$ or  equivalently, by (ii), the open set $V\not=\emptyset$. Since $p_0$ is  topologically transitive, $p_0{\cdot}t\in V$ for some $t\in \R$, which is absurd. In consequence, $P_0=R$, $V=\emptyset$ and $P=F$.
\par
(iv) If $P_0=\emptyset$, then $F=\emptyset$. Else $P=F$ and $V=\emptyset$ by (iii),  since any point in $P$ is topologically transitive when $P$ is minimal.
\par
(v) If $P=V$, we are done. Else, $\emptyset\not=F\subsetneq P$ is a closed set and since $\supp(\mu)=P$, $\mu(V)=\mu(P\setminus F)>0$. By the invariance of $V$ and the ergodicity of $\mu$, necessarily $\mu(V)=1$ and $\mu(F)=0$. Finally, if $V$ were not dense, for some $p\in F$ there would be a ball $B(p,r)\subset F$ ($r>0$), and this cannot happen if  $\supp(\mu)=P$.
The proof is finished.
\end{proof}
In the description of the attractor $\A$ when $P$ is minimal and uniquely ergodic and the upper Lyapunov exponent is null (see~\cite{calaobsa}), only in the linear-dissipative case it can happen that $b(p)\gg 0$ for any $p\in P$, and then $b$ is continuous on $P$ and thus, uniformly strongly positive. But note that in the general case the upper boundary map $b$ is only known to be semicontinuous, and it can happen that $b(p)\gg 0$ for any $p\in P$, and $b$ is not uniformly strongly positive: see Example~\ref{eje p_0}~(ii), where this behaviour is shown for a linear-dissipative family.
In the next result we give a sufficient condition in order to have $b$ uniformly strongly above $0$.
\begin{prop}\label{prop 3.5}
Assume conditions $(c1)$-$(c3)$ on $g$, assume that $b(p)\gg 0$ for any $p\in P$ and let $R$ be the set of continuity points for $b$. Then, if $R\cap M\not=\emptyset$ for any minimal set $M\subset P$, there exists a $z_0\gg 0$ such that $b(p)\geq z_0$ for any $p\in P$.
\end{prop}
\begin{proof}
Let $(M_i)_{i\in I}$ denote the family of minimal sets for the flow in $P$, for an appropriate index set $I$. Assuming the axiom of choice, we can take a $p_i\in M_i$ for each $i\in I$. Now, for each family $(\delta_i)_{i\in I}$ of positive numbers $\delta_i>0$, we affirm that
$P=\cup_{i\in I} \cup_{t\geq 0} B(p_i,\delta_i){\cdot}t$, where $B(p_i,\delta_i)$ denotes the open ball centered at $p_i$ with radius $\delta_i$. The argument is simple: if $p\in P$, we consider its $\alpha$-limit set $F_p$, which is a compact invariant set and thus contains a minimal set $M_i$ for some $i\in I$. Then, for the element $p_i\in M_i\subset F_p$, there exists a sequence $(t_n)_n\downarrow -\infty$ such that $p{\cdot}t_n\to p_i$ as $n\to \infty$. Therefore, there exists an integer $n_0$ such that for $n\geq n_0$, $p{\cdot}t_n\in B(p_i,\delta_i)$ or equivalently,  $p\in B(p_i,\delta_i){\cdot}(-t_n)$ with $-t_n>0$.
\par
At this point, fix a $z\gg 0$ and note that we can choose the points $p_i\in R\cap M_i$, $i\in I$ so that they are points of continuity for $b$. Then, for each $i\in I$, since $b(p_i)\gg 0$, we can take $\delta_i>0$ and $\varepsilon_i>0$ such that $b(p)\geq \varepsilon_i z$ for any $p\in B(p_i,\delta_i)$. Since $P=\cup_{i\in I} \cup_{t\geq 0} B(p_i,\delta_i){\cdot}t$, the sets $B(p_i,\delta_i){\cdot}t$ are open because the maps $\sigma_t:P\to P$, $p\mapsto p{\cdot}t$ are homeomorphisms, and $P$ is compact, we can extract a finite covering which we can write for simplicity by  $P=\cup_{k=1}^{n_1} \cup_{j=1}^{n_2} B(p_{{i_k}},\delta_{{i_k}}){\cdot}t_j$. Now, since the semiflow is strongly monotone, for any $p\in B(p_{{i_k}},\delta_{{i_k}})$, $b(p{\cdot}t_j)=u(t_j,p,b(p))\geq u(t_j,p,\varepsilon_{{i_k}} z)\geq \wit \varepsilon_{{i_k}j} z$ for a certain  $\wit \varepsilon_{{i_k}j}>0$. To conclude the proof, it suffices to take $z_0=\varepsilon z\gg 0$ for $\varepsilon=\min\{\wit \varepsilon_{{i_k}j}\mid 1\leq k\leq n_1,\,1\leq j\leq n_2 \}>0$.
\end{proof}
\subsection{The linear equations: continuous separation, principal spectrum and Lyapunov exponents}\label{sec-lineal}
Here we collect some fundamental facts related to the associated linear equations. Recall that $P$ is a compact metric space with a continuous flow on it. Let us denote by $\tau_L$ the  linear skew-product semiflow
\begin{equation*}%\label{tauL}
\begin{array}{cccl}
 \tau_L: & \R_+\times P\times X& \longrightarrow & \hspace{0.3cm}P\times X\\
 & (t,p,z) & \mapsto
 &(p{\cdot}t,\phi(t,p)\,z)\,
\end{array}
\end{equation*}
induced by  the mild solutions $\phi(t,p)\,z$ of the linear ACPs associated to the family of linearized problems along the null solution of~\eqref{pdefamilynl},  which with condition (c1) is exactly given by the linear part of the problems:
\begin{equation}\label{linealizada}
\left\{\begin{array}{l} \des\frac{\partial y}{\partial t}  =
 \Delta \, y+h(p{\cdot}t,x)\,y\,,\quad t>0\,,\;\,x\in U, \;\, \text{for each}\; p\in P,\\[.2cm]
By:=\bar\alpha(x)\,y+\kappa\,\des\frac{\partial y}{\partial n} =0\,,\quad  t>0\,,\;\,x\in \partial U.
\end{array}\right.
\end{equation}
In particular $\phi(t,p)$ are bounded operators on $X$ which are compact for $t>0$  and satisfy the linear semicocycle property $\phi(t+s,p)=\phi(t,p{\cdot}s)\,\phi(s,p)$, $p\in P$, $t,s\geq 0$.
\par
In the case of Dirichlet boundary conditions, as mentioned before,  we can also build  a linear skew-product semiflow $\tau_L$ on the space $P\times X^\alpha$. Doing this, we gain the strong positivity of the operators $\phi(t,p)$ for $t>0$ and $p\in P$. Together with the compact character of the linear operators, this implies that the linear skew-product semiflow  $\tau_L$ admits a continuous separation $X^\gamma=X_1(p)\oplus X_2(p)$ for  $p \in P$, in the terms stated in Pol\'{a}\v{c}ik and Tere\v{s}\v{c}\'{a}k~\cite{pote} and Shen and Yi~\cite{shyi}. Namely, there are two families of subspaces $\{X_1(p)\}_{p\in P}$ and $\{X_2(p)\}_{p\in P}$ of $X^\gamma$ which satisfy:
\begin{itemize}
\item[(1)] $X^\gamma=X_1(p)\oplus X_2(p)$  and $X_1(p)$, $X_2(p)$ vary
    continuously in $P$;
 \item[(2)] $X_1(p)=\langle e(p)\rangle$, with $e(p)\gg 0$ and
     $\|e(p)\|=1$ for any $p\in P$;
\item[(3)] $X_2(p)\cap X^\gamma_+=\{0\}$ for any $p\in P$;
\item[(4)] for any $t>0$,  $p\in P$,
\begin{align*}
\phi(t,p)\,X_1(p)&= X_1(p{\cdot}t)\,,\\
\phi(t,p)\,X_2(p)&\subset X_2(p{\cdot}t)\,;
\end{align*}
\item[(5)] there are $M_0>0$, $\delta_0>0$ such that for any $p\in P$, $z\in
    X_2(p)$ with $\|z\|_\gamma=1$ and $t>0$, $\|\phi(t,p)\,z\|_\gamma\leq M_0 \,e^{-\delta_0 t}\|\phi(t,p)\,e(p)\|_\gamma$.
\end{itemize}
\par
In this situation,  the 1-dim invariant subbundle $\bigcup_{p\in P} \{p\} \times X_1(p)$ is called the {\em principal bundle\/}. Note that in the case of Dirichlet boundary conditions the vectors $e(p)\in\Int X^\alpha_+$ generating the principal bundle are normalized with respect to the sup-norm.   The Sacker-Sell spectrum or continuous spectrum of the restriction of $\tau_L$ to this invariant subbundle is called the {\em principal spectrum\/} of $\tau_L$, and is denoted by  $\Sigma_{\text{pr}}$ (see Mierczy{\'n}ski and Shen~\cite{mish}). The classical theory of Sacker and Sell~\cite{sase} for finite-dimensional linear skew-product flows says that, provided that $P$ is connected, the principal spectrum $\Sigma_{\text{pr}}$ is a possibly degenerate compact interval of the real line, that is, $\Sigma_{\text{pr}}=[\alpha_P,\lambda_P]$ with $\alpha_P\leq\lambda_P$. The upper end $\lambda_P$ is the upper Lyapunov exponent of $\tau_L$.
It is well-known that if $P$ is uniquely ergodic, then $\Sigma_{\text{pr}}=\{\lambda_P\}$. In what follows, we assume that $P$ is compact and connected.
\par
Besides, the continuous separation permits to associate to $h$ the 1-dim continuous linear cocycle  $c(t,p)$ driving the linear dynamics inside  the principal bundle, that is, given for $t\geq 0$ by the positive number $c(t,p)$ such that
\begin{equation}\label{c}
\phi(t,p)\,e(p)=c(t,p)\,e(p{\cdot}t)\,,\quad t\geq 0\,,\; p\in P
\end{equation}
and by the relation $c(-t,p)=1/c(t,p{\cdot}(-t))$ for any $t>0$ and $p\in P$. This 1-dim linear cocycle has certainly proved to be a fundamental tool in the dynamical study of the infinite dimensional problems we are dealing with (see~\cite{calaobsa} and~\cite{laos}).
\par
In what refers to Lyapunov characteristic exponents for the linear dynamics of $\tau_L$, we recall the standard definitions for $p\in P$ and $z\in X^\gamma$, $z\not= 0$:
\[
\lambda_s^+(p,z):= \limsup_{t\to \infty} \frac{\ln \|\phi(t,p)\,z\|_\gamma}{t}\;\; \text{and}\;\; \lambda_{i}^+(p,z):= \liminf_{t\to \infty} \frac{\ln \|\phi(t,p)\,z\|_\gamma}{t}\,.
\]
The subscripts $s$ and $i$ stand for superior or inferior limits, respectively, and the superscript $+$ means that $t\to\infty$.
In the case of Dirichlet boundary conditions, it is important to know that one can  argue as in Proposition~3 in Obaya and Sanz~\cite{obsa2019} to see that the previous exponents can be equally computed using the sup-norm instead of the norm in $X^\alpha$.
Besides, whenever it makes sense, that is, if we are dealing with a linear flow instead of a semiflow, one also defines the exponents $\lambda_s^-(p,z)$ and $\lambda_{i}^-(p,z)$ by taking limits as $t\to -\infty$. For instance, it makes sense to consider the four Lyapunov characteristic exponents linked to the 1-dim linear cocycle  $c(t,p)$, which in this scalar case are independent of $z\in \R$, $z\not= 0$:
\begin{equation*}%\label{lyap}
\lambda_s^{\pm}(p):=\limsup_{t\to \pm\infty} \frac{\ln c(t,p)}{t} \quad\text{ and }\quad \lambda_{i}^{\pm}(p):=\liminf_{t\to \pm\infty} \frac{\ln c(t,p)}{t}\,.
\end{equation*}
In particular, the Lyapunov exponents $\lambda_s^{\pm}(p)$  and $\lambda_{i}^{\pm}(p)$ for $p\in P$ lie in $[\alpha_P,\lambda_P]$ (see~\cite{sase}). Besides, by Theorem~2.3 in Johnson et al.~\cite{jops} there are ergodic measures $\nu$ and $\mu$ for the flow on $P$ (they may coincide) such that
\begin{equation}\label{medidas erg}
\begin{split}
\alpha_P=&\lim_{t\to\infty} \frac{\ln c(t,p)}{t}=\lim_{t\to-\infty} \frac{\ln c(t,p)}{t}\quad  \text{$\nu$-a.e.},\\
\lambda_P=&\lim_{t\to\infty}  \frac{\ln c(t,p)}{t}=\lim_{t\to -\infty}  \frac{\ln c(t,p)}{t}\quad  \text{$\mu$-a.e.}.
\end{split}
\end{equation}
\par
On the other hand, the  pullback Lyapunov characteristic exponents are defined as follows for linear semiflows over a base flow, denoted with a superscript $'$, and for simplicity we write the definition directly for the sup-norm:
\begin{equation*}
\lambda_{s}'(p,z):= \limsup_{t\to \infty} \frac{\ln \|\phi(t,p{\cdot}(-t))\,z\|}{t}\,, \quad \lambda_{i}'(p,z):= \liminf_{t\to \infty} \frac{\ln \|\phi(t,p{\cdot}(-t))\,z\|}{t}\,.
\end{equation*}
\par
Due to the existence of a continuous separation, we can calculate the forwards and pullback characteristic exponents of pairs $(p,z)$ with $z\gg 0$ in terms of the characteristic exponents of the 1-dim cocycle $c(t,p)$ as $t\to \infty$ or $t\to -\infty$, respectively. The result is stated for the exponents $\lambda_s^+(p,z)$ and $\lambda_{s}'(p,z)$, but it can immediately be rephrased for $\lambda_{i}^+(p,z)$ and $\lambda_{i}'(p,z)$ just by taking inferior limits instead.
\begin{prop}\label{prop-exponentes Lyap}
For any $p\in P$ and any $z\gg 0$,
\[
\lambda_s^+(p,z)=\limsup_{t\to \infty} \frac{\ln c(t,p)}{t}=:\lambda_s^{+}(p) \;\; \text{and}\;\;\;\lambda_{s}'(p,z) =\limsup_{t\to -\infty} \frac{\ln c(t,p)}{t}=:\lambda_s^{-}(p)\,. \]
\end{prop}
\begin{proof}
Fix a $z\in X^\gamma$, $z\gg 0$. By the properties of the continuous separation and the compact character of $P$, we can find constants $c_1, c_2>0$ such that $c_1z\leq e(p)\leq c_2z$ for any $p\in P$. Applying monotonicity of both $\tau_L$ and the norm, on the one hand we have that $c_1\|\phi(t,p)\,z\| \leq c(t,p)\leq c_2\|\phi(t,p)\,z\|$, from where the formula for $\lambda_s^+(p,z)$ follows. On the other hand, we obtain that $c_1 \|\phi(t,p{\cdot}(-t))\,z\| \leq c(t,p{\cdot}(-t)) \leq c_2 \|\phi(t,p{\cdot}(-t))\,z\|$, so that
\[
\lambda_{s}'(p,z)=\limsup_{t\to \infty} \frac{\ln \|\phi(t,p{\cdot}(-t))\,z\|}{t}= \limsup_{t\to \infty} \frac{\ln c(t,p{\cdot}(-t))}{t} =\limsup_{t\to \infty} \frac{\ln c(-t,p)}{-t}\,,
\]
by the cocycle property of $c(t,p)$. The proof is finished.
\end{proof}
%%%%%%%%%%%%%%%%%%%%%%%%%%%%%%%%%%%%%%%%%%%%%%%%%%%%%%%%%%%%%%%%%%%%%%%%%%%%%%%%%%%%%%%%%%%%%%%%%%%%%%%%%%%%%%%%%%
%%%%%%%%%%%%%%%%%%%%%%%%%%%%%%%%%%%%%%%%%%%%%%%%%%%%%%%%%%%%%%%%%%%%%%%%%%%%%%%%%%%%%%%%%%%%%%%%%%%%%%%%%%%%%%%%%%%%%%%%%%%%%
\section{Linear-dissipative problems over a compact base flow}\label{sect-linear dissipative}\noindent
In this section we concentrate on the linear-dissipative case, that is, for the family of problems~\eqref{pdefamilynl} over a compact and connected base flow $P$ we assume the existence of a zone around $0$   where the dissipative term is negligible and the problems are linear. In this case the dynamical behaviour of the linear part has strong implications in the general nonlinear dynamics. We keep the notation and terminology introduced in the previous section.
\par
First of all, the 1-dim continuous linear cocycle  $c(t,p)$ in~\eqref{c} permits to characterize the nontrivial sections $A(p)$ of the attractor, that is, when $b(p)\gg 0$. Another characterization of a pullback nature is offered, also in terms of the linear dynamics.
\begin{prop}\label{prop-equivalentes}
Assume conditions $(c1)$-$(c5)$ on $g$ with $r_0>0$ in $(c5)$, and let us fix a $z_0\in X^\gamma$, $z_0\gg 0$. Then, given a $p\in P$, the following conditions are equivalent:
\begin{itemize}
\item[(i)] $b(p)\gg 0$;
\item[(ii)] $\des\sup_{t\leq 0} c(t,p)<\infty$;
\item[(iii)] $\des\inf_{t\geq 0}\|\phi(t,p{\cdot}(-t))\,z_0\|>0$.
\end{itemize}
\end{prop}
\begin{proof}
The proof of (i)$\Leftrightarrow$(ii) can be extracted from that of Proposition~5.3 in Caraballo et al.~\cite{calaobsa}. To prove~(ii)$\Leftrightarrow$(iii), as done in the proof of Proposition~\ref{prop-exponentes Lyap}, we take constants $c_1, c_2>0$ such that $c_1z_0\leq e(p)\leq c_2z_0$ for any $p\in P$. Then, applying the monotonicity of $\tau_L$, for any $t\geq 0$,
$c_1\phi(t,p{\cdot}(-t))\,z_0 \leq \phi(t,p{\cdot}(-t))\,e(p{\cdot}(-t)) \leq c_2 \phi(t,p{\cdot}(-t))\,z_0$, and then
$c_1\phi(t,p{\cdot}(-t))\,z_0 \leq c(t,p{\cdot}(-t))\,e(p) \leq c_2 \phi(t,p{\cdot}(-t))\,z_0$  by~\eqref{c}. Now, since the sup-norm is monotone, for any $t\geq 0$,
\[
c_1\|\phi(t,p{\cdot}(-t))\,z_0\| \leq c(t,p{\cdot}(-t))=\frac{1}{c(-t,p)} \leq c_2 \|\phi(t,p{\cdot}(-t))\,z_0\|\,,
\]
from where the equivalence follows immediately. The proof is finished.
\end{proof}
\begin{nota}\label{nota-1 purely}
It is important to note that $\sup_{t\leq 0} c(t,p)<\infty$ is a necessary condition for  $b(p)\gg 0$ independently of whether $r_0>0$ or $r_0=0$ in condition (c5) (once more, see the proof of Proposition~5.3 in~\cite{calaobsa}). However, it is not a sufficient condition in the purely dissipative case. We will return to this issue  in Section \ref{sect-purely dissipative}.
\end{nota}
Next we study some structural properties of the attractor in terms of the location of the principal spectrum $\Sigma_{\text{pr}}=[\alpha_P,\lambda_P]$ in the real line. The situations considered  correspond to the idea of the principal spectrum being pulled from left to right in the real line. More precisely, for $\alpha_P<0$ we distinguish the three possible situations (s1) with $\lambda_P <0$, (s2) with $\lambda_P =0$ and (s3) with $\lambda_P >0$;  $\alpha_P=0$ in situation (s4) and $\alpha_P>0$ in situation (s5).  Note that the case $\Sigma_{\text{pr}}=\{0\}$ fits in (s4). The exponential convergence to $0$ when $\lambda_P < 0$ and the uniform persistence when $\alpha_P>0$ are standard results in more general settings: see Mierczy{\'n}ski and Shen~\cite{mish08} and Novo et al.~\cite{noos7}. Anyway, statements and proofs are included for the sake of completeness.
\par
Ergodic measures are always meant with respect to the flow on $P$.
\begin{teor}\label{teor-principal spectrum}
Assume conditions $(c1)$-$(c5)$ on $g$ with $r_0>0$ in $(c5)$. Then:
\begin{itemize}
\item[(s1)] If $\alpha_P\leq\lambda_P < 0$, then $b\equiv 0$ and $\A$ is the trivial set $P\times \{0\}$. Moreover, $\A$ is uniformly exponentially stable, that is, for any  $0<\varepsilon< |\lambda_P|$ there is a $C_\varepsilon>0$ such that $\|u(t,p,z)\|\leq C_\varepsilon\, e^{(\lambda_P+\varepsilon)t}\|z\|$ for any $t\geq 0$, $p\in P$ and $z\in X$.
\item[(s2)]  If $\alpha_P<0=\lambda_P$, then,  on the one hand (because $\alpha_P<0$) there exists an ergodic measure $\nu$ such that $b(p)=0$ for almost every $p$ with respect to $\nu$, and on the other hand  there exists  at least one $p\in P$ such that $b(p)\gg 0$, so that the attractor is nontrivial.
\item[(s3)] If  $\alpha_P < 0 <\lambda_P$, then there exists an ergodic measure $\nu$ as in {\rm(s2)}, and (because $\lambda_P>0$) there exists an ergodic measure $\mu$ such that $b(p)\gg 0$, $\lim_{t\to \infty} c(t,p)=\infty$ and $\limsup_{t\to \infty} \|b(p{\cdot}t)\|\geq r_0$ for almost every $p$ with respect to $\mu$, and the attractor $\A$ is nontrivial.
\item[(s4)] If  $\alpha_P = 0 \leq \lambda_P$, then in every minimal set $M$ in $P$ there is some $p\in M$ such that  $b(p)\gg 0$ and the attractor is nontrivial. If besides  $\lambda_P>0$, there is an ergodic measure $\mu$ as in {\rm(s3)}.
\item[(s5)] If $0<\alpha_P\leq  \lambda_P$, then:
\begin{itemize}
\item[(i)] $b$ is uniformly strongly positive, i.e., there is a $z\gg 0$ such that $b(p)\geq z$ for $p\in P$. Also $\lim_{t\to \infty} c(t,p)=\infty$ and $\limsup_{t\to \infty} \|b(p{\cdot}t)\|\geq r_0$ for any $p\in P$. Besides, the attractor contains a ruled surface inside the principal bundle; namely, $\{(p,re(p))\mid p\in P,\, 0<r\leq r_1 \}\subset \A\cap (P\times \Int X^\gamma_+)$ for some $r_1>0$.
\item[(ii)] The semiflow $\tau$ is  uniformly persistent in the interior of the positive cone $\Int X^\gamma_+$; namely, there is a $z_0\gg 0$ such that for  any $z\gg 0$ there exists a $t^*=t^*(z)$ such that $u(t,p,z)\geq z_0$ for any $p\in P$ and $t\geq t^*$.
\end{itemize}
\end{itemize}
\end{teor}
\begin{proof}
(s1)  The uniform exponential stability has been proved in Proposition~5 in Cardoso et al.~\cite{cardoso}. Just note that if $\lambda_P<0$, the continuous spectrum for the linear semiflow $\tau_L$ lies strictly on the negative real semiaxes. By \eqref{b(p)}, it must be $b\equiv 0$.
\par
(s2) By \eqref{medidas erg}, we can  take an ergodic measure $\nu$  such that
$\lim_{t\to -\infty} \ln c(t,p)/t=\alpha_P<0$ for almost every $p$ with respect to  $\nu$. Then,   $\lim_{t\to -\infty} c(t,p)=\infty$ for all such $p$, and according to Proposition~\ref{prop-equivalentes}, $b(p) = 0$.
\par
On the other hand, we now pay attention to the value $\lambda_P=0$. If for some $p_0\in P$ there is a nontrivial bounded orbit for the linear skew-product flow $L$ on $P\times \R$ determined by the linear 1-dim cocycle $c(t,p)$, or in other words, if $\{c(t,p_0)\mid t\in\R\}$ is bounded, then  $b(p_0)\gg 0$ by Proposition~\ref{prop-equivalentes}. So, let us argue by contradiction and let us assume that there are no nontrivial bounded orbits for $L$, that is,
\[
\mathcal{B}=\{(p,r)\in P\times \R\,\,/\,\, |c(t,p)r|\; \text{is uniformly bounded in}\; t\}=P\times\{0\}\,.
\]
Then, for any minimal set $M\subset P$, a classical result by Selgrade~\cite{selg} says that $L$ admits an exponential dichotomy over $M$. Besides, in this case, since the continuous spectrum over $M$, $\Sigma(M)\subset \Sigma(P)=[\alpha_P,0]$,   necessarily all the sections of the stable set $\mathcal{S}(p)$ are one-dimensional  as well as all the sections of the unstable set are trivial, i.e., $\mathcal{U}(p)=\{0\}$ for every $p\in M$.
%Lemma 6 and Lemma 7 in Sacker and Sell 1978
Since the same happens for any minimal set, one can apply Lemma~12 and Theorem~2 in Sacker and Sell~\cite{sase76} to conclude that $L$ admits an exponential dichotomy over $P$, that is, $0$ belongs to the resolvent set of $L$. But this is a contradiction with the fact that $\lambda_P=0$ lies in the continuous spectrum, and we are done.
\par
(s3) The measure $\nu$ is provided by $\alpha_P<0$ as in (s2). Now, for $\lambda_P> 0$ we can  take an ergodic measure $\mu$ such that
$\lim_{t\to \pm\infty} \ln c(t,p)/t=\lambda_P>0$ for almost every $p$ with respect to  $\mu$. Then, on the one hand  $\lim_{t\to -\infty} c(t,p)=0$ for all such $p$, and according to Proposition~\ref{prop-equivalentes}, $b(p)\gg 0$. On the other hand, $\lim_{t\to\infty}c(t,p)=\infty$ for all such $p$. At this point one can argue as in the proof of Proposition~4.14~(b) in Langa et al.~\cite{laos} to affirm that there exists a sequence $(t_n)_n\uparrow \infty$ (which depends on $p$) such that $\|b(p{\cdot}t_n)\|>r_0$ for any $n\geq 1$, so that in particular $\limsup_{t\to \infty} \|b(p{\cdot}t)\|\geq r_0$ for almost every $p$ with respect to $\mu$.
\par
(s4) Maintaining the notation in the proof of (s2), fixed a minimal set $M\subseteq P$ the continuous spectrum of $L$ over $M$ satisfies $\Sigma(M)=[\alpha_M,\lambda_M]\subseteq [0,\lambda_P]$. It might be  $\alpha_M=0$. Then, $L$ does not have an exponential dichotomy over $M$, and since $M$ is minimal the aforementioned result by Selgrade~\cite{selg} implies the existence of a nontrivial bounded orbit for $L$, that is, $\{c(t,p)\mid t\in \R\}$ is bounded for some $p\in M$, and then   $b(p)\gg 0$ by Proposition~\ref{prop-equivalentes}. Otherwise  $\alpha_M>0$. Then, there is an ergodic measure $\nu_M$ over $M$ such that $\lim_{t\to -\infty} \ln c(t,p)/t=\alpha_M>0$ for almost every $p\in M$ with respect to  $\nu_M$ and once more by Proposition~\ref{prop-equivalentes} $b(p)\gg 0$ for all such $p$. Finally, we can argue as in (s3) for $\lambda_P> 0$.
\par
(s5) We now assume that the principal spectrum is contained in the set of positive real numbers.
(i) First of all, since $\Sigma_{\text{pr}} =[\alpha_P,\lambda_P]\subset (0,\infty)$, $0$ lies in the resolvent set of the linear skew-product semiflow $L$ on $P\times \R$, that is, it has an exponential dichotomy over $P$, in fact with full unstable subspace and null stable subspace. More precisely, there exist $C>0$ and $\beta>0$ such that
\begin{equation*}%\label{de}
c(t,p)\,c(s,p)^{-1}\leq C\,e^{-\beta (s-t)}\quad \text{for} \;p\in P\; \text{and}\; t\leq s\,.
\end{equation*}
From this we first get that $\sup\{c(t,p)\mid p\in P,\,t\leq 0\}\leq C<\infty$, so that  $b(p)\gg 0$ for any $p\in P$ by Proposition~\ref{prop-equivalentes}.
Second, by taking $t=0$ and  $c_0=1/C$, we get that $c(t,p)\geq c_0\, e^{\beta t}$ for $t\geq 0$ and $p\in P$,  so that $\lim_{t\to\infty} c(t,p)=\infty$ for any $p\in P$. Then, $\limsup_{t\to \infty} \|b(p{\cdot}t)\|\geq r_0$ for all $p\in P$, as in (s3).
\par
Moreover, taking $r_1=r_0/C>0$, we have that $r_1 c(t,p)\,e(p{\cdot}t)(x)\leq r_0$ for any $p\in P$, $t\leq 0$ and $x\in\bar U$, and the nontrivial segment $\{(p,re(p))\mid 0<r\leq r_1\}\subset \A\cap (P\times \Int X^\gamma_+)$: just note that for $0<r\leq r_1$, the map $z(s)=rc(s,p)\,e(p{\cdot}s)$ for $s\leq 0$ provides a backward semiorbit for $\tau_L$, i.e., $\phi(t,p{\cdot}s)\,z(s)=z(t+s)$ for $s\leq 0$ and  $0\leq t\leq -s$. Besides, it remains in the zone where the problems are linear. Thus, it is a solution of the nonlinear problem too, and we can continue this solution to have a bounded entire orbit, so that $(p,re(p))\in\A$. Note that in this situation, $r_1 e(p)\leq b(p)$ for any $p\in P$ and, as in the proof of Proposition~\ref{prop-exponentes Lyap}, we can find a $z\gg 0$ such that $b(p)\geq z$ for any $p\in P$.
\par
(ii) Let us fix a $z\gg 0$ and, once more as in the proof of Proposition~\ref{prop-exponentes Lyap}, take constants $c_1, c_2>0$ such that $c_1z\leq e(p)\leq c_2z$ for any $p\in P$. Then, by the monotonicity of $\tau_L$, for $t\geq 0$,
\[
\phi(t,p)\,z\geq \frac{1}{c_2}\,\phi(t,p)\,e(p) =\frac{1}{c_2}\,c(t,p)\,e(p{\cdot}t)\geq \frac{c_1}{c_2}\,c(t,p)\,z\geq \frac{c_1}{c_2}\,c_0\, e^{\beta t}\,z\,,
\]
and taking $r>0$ big enough so that $\frac{c_1}{c_2}\,c_0\, e^{\beta r}> 2$, we get that $\phi(r,p)\,z \gg 2z$ for any $p\in P$. To finish, just note that this is the condition required in Theorem~3.3 in~\cite{noos7} (taking the compact invariant set  $K=P\times\{0\}\equiv P$) to guarantee the uniform persistence of $\tau$ in the interior of the positive cone. In fact, the proof of the cited theorem permits to assert that the time $t^*(z)$ in the statement can be taken common for all $p\in P$.
\end{proof}
\begin{notas}\label{nota-2 purely}
 1. Note that some of the arguments used in the proof identically work  in the purely dissipative case, that is, when $r_0=0$ in condition (c5). More precisely, the proofs of items~(s1) and~(s5.ii) as well as the argument developed for $\alpha_P<0$ are independent of the linear-dissipative structure.
\par
 2. In situation (s2), with $\alpha_P<0=\lambda_P$, there always coexist elements $p$ where $b(p)=0$ with others where $b(p)\gg 0$. In contrast with this, in situation (s4) with $\alpha_P=0\leq \lambda_P$, it might be $b(p)\gg 0$ for any $p\in P$. The reader is referred to Example~\ref{eje a_0}~(ii). For more examples,  recall that when $P$ is  minimal and uniquely ergodic and $\alpha_P=\lambda_P=0$, either $\A$ is a wide attractor, meaning that $b(p)\gg 0$ for any $p\in P$, or it is a pinched compact set and $b$ vanishes over an invariant residual set of either null or full measure (see Theorems~5.1 and~5.2 in~\cite{calaobsa}), often showing a chaotic behaviour in a precise sense  (see Theorem~5.9  in~\cite{calaobsa}).
 \par
3. Related to the last comment, assume that $\alpha_P\leq 0\leq \lambda_P$. Assume also that there exists an ergodic measure $\nu$ with associated Lyapunov exponent $\lambda=0$, that is, $\lim_{t\to\pm\infty} \ln c(t,p)/t=0$ for almost every $p$ with respect to $\nu$. This happens for sure if $\alpha_P= 0$ or if $\lambda_P=0$: see \eqref{medidas erg}. In line with a classical result by Shneiberg~\cite{shne}, it can be deduced that for almost every $p\in P$ with respect to $\nu$ there exists a sequence $(t_n)_n\uparrow \infty$, which depends on $p$, such that $c(t_n,p)=1$ for $n\geq 1$  ($p$ is said to be recurrent at $\infty$). In fact the set of recurrent points $P_{\rm{r}}$, which contains the points $p \in P$ such that $p{\cdot}t$ is recurrent at $\pm\infty$ for every $t \in \R$, is invariant and satisfies that $\nu(P_{\rm{r}})=1$. For all the details, see Proposition~5.5 in~\cite{calaobsa}.  Arguing as in the proof of Theorem~4.17~(a) in~\cite{laos}, if there is a $p\in P_{\rm{r}}$ with $b(p)\gg 0$, then the trajectory of $(p,b(p))$ inside the attractor remains in the zone where the problems are linear, that is,  $0\ll b(p{\cdot}t)\leq r_0$ for $t\in \R$, meaning that $b(p{\cdot}t)(x)\leq r_0$ for $t\in\R$ and $x\in\bar U$, and besides $\limsup_{t\to\infty} \|b(p{\cdot}t)\|=r_0$. We note that the set $P_{\rm{r}}$ has played a fundamental role in the study of the dynamics of linear-dissipative problems in the papers~\cite{calaobsa} and~\cite{laos}.
\par
On the other hand, if $p\in P_{\rm{r}}$ with $b(p)\gg 0$, then the section of the attractor  $A(p)$ is contained in the principal bundle, that is, $A(p)\subset X_1(p)$. To see it, take a $z\in A(p)\subseteq [-b(p),b(p)]$. Then,  the full orbit through $(p,z)$, written for simplicity as $(p{\cdot}t,z(t))$ for $t\in\R$, lies in the zone where the problems are linear, so that it is a full orbit for the linear semiflow $\tau_L$ as well. Recalling the properties of the continuous separation, we can decompose $z(t)=z_1(t)+z_2(t)\in X_1(p{\cdot}t)\oplus X_2(p{\cdot}t)$, $t\in\R$ and by the invariance property, $z_2(0)=\phi(t,p{\cdot}(-t))\,z_2(-t)$ and $z_2(-t)$ is bounded for $t\geq 0$, say  $\|z_2(-t)\|_\gamma\leq c_0$ for $t\leq 0$.  Then, using property (5) and relation~\eqref{c}, $\|z_2(0)\|_\gamma=\|\phi(t,p{\cdot}(-t))\,z_2(-t)\|_\gamma\leq \|z_2(-t)\|_\gamma M_0 \,e^{-\delta_0 t}\|\phi(t,p{\cdot}(-t))\,e(p{\cdot}(-t))\|_\gamma \leq  c_0\,M_0 \,e^{-\delta_0 t} c(t,p{\cdot}(-t))\,\|e(p)\|_\gamma$ for $t>0$. Here, since $p\in P_{\rm{r}}$, we can take a sequence $(t_n)\uparrow \infty$ such that $c(t_n,p{\cdot}(-t_n))=1/c(-t_n,p)=1$ for $n\geq 1$. Then, we can conclude that $\|z_2(0)\|_\gamma=0$ and thus $z=z_1(0)\in X_1(p)$, as we wanted. After this discussion, we affirm that for $p\in P_{\rm{r}}$ with $b(p)\gg 0$ and $\liminf_{t\to\infty} \|b(p{\cdot}t)\|=0$, the pullback attractor $\{A(p{\cdot}t)\}_{t\in  \R}$ for the process $S_p(\cdot,\cdot)$ generated  by the solutions of the non-autonomous parabolic equation obtained by the evaluation  of the coefficients along the trajectory of $p$ is chaotic in the sense of Li-Yorke: see Definition~4.22 and Theorem~4.23 in~\cite{laos} for the precise terms.
\par
Besides, if $P$ does not contain any minimal and periodic subsets,  then following the methods mentioned in N\'{u}\~{n}ez and Obaya~\cite{nuob} it is possible to find a map $a\in C(P)$ with $\int_P a\,d\nu=0$ such that the family of problems~\eqref{pdefamilynl} over  $P$ with linear coefficient $h(p{\cdot}t,x)+a(p{\cdot}t)$ has a global attractor which is fiber-chaotic in measure with respect to $\nu$ in the sense of Li-Yorke. If $\nu_1$ and $\nu_2$ are two ergodic measures on $P$ with null Lypunov exponent and $\supp(\nu_1)\cap \supp(\nu_2)\not=\emptyset$, the previous construction is congruent in the sense that it is possible to choose the same map $a\in C(P)$ for both measures. The above arguments justify  the presence in some cases of a  big set of dynamical unpredictability in the phase space. Even, when $\Sigma_{\text{pr}}=\{0\}$, this could be a set of
complete measure.
\end{notas}
We finish this section by adding the strict sublinearity condition (c6) on $g$, which permits to substantially extend the dynamical information that we have in some cases, namely, when $\lambda_P>0$ or $\alpha_P>0$. With this extra condition, a standard argument of comparison of solutions shows that the semiflow $\tau$ is sublinear in the positive cone, that is,  $u(t,p,\lambda\, z)\leq \lambda\, u(t,p,z)$ for  $\lambda\geq 1$, $p\in P$, $z\geq 0$ and $t\geq 0$.
\par
First of all, we state a preliminary result in the sublinear setting which gives conditions on $p_0$ so that in the positive cone, every orbit starting below  or above $b(p_0)$ approaches $b(p_0{\cdot}t)$ as $t\to \infty$.
\begin{teor}\label{teor-sublineal lin-dis}
Assume conditions $(c1)$-$(c6)$ on $g$, with $r_0>0$ in $(c5)$ and $(c6)$. If for some $p_0\in P$ we have that $b(p_0)\gg 0$ and $\limsup_{t\to\infty} c(t,p_0)=\infty$, then:
\begin{itemize}
\item[(i)] For any $0< z\leq b(p_0)$ it holds that $\limsup_{t\to \infty} \|u(t,p_0,z)\|\geq r_0$ and $\des\lim_{t\to\infty} b(p_0{\cdot}t)-u(t,p_0,z)=0$.
\item[(ii)] If $z\gg 0$ is such that $b(p_0)\leq z$, then $\des\lim_{t\to\infty} u(t,p_0,z)-b(p_0{\cdot}t)=0$.
\end{itemize}
\end{teor}
\begin{proof}
If $b(p_0)\gg 0$ and $\limsup_{t\to\infty} c(t,p_0)=\infty$, then $\limsup_{t\to \infty} \|b(p_0{\cdot}t)\|\geq r_0$: once more, see the proof of Proposition~4.14~(b) in Langa et al.~\cite{laos}. Now, for $0< z\leq b(p_0)$, note that by the strong monotonicity $u(t,p_0,z)\gg 0$ for any $t>0$. For this reason we can assume without loss of generality that $0\ll z\leq b(p_0)$ and the argument to prove that $\limsup_{t\to \infty} \|u(t,p_0,z)\|\geq r_0$ is just the same as before.
\par
For the rest of the proof, the reader is referred to the proof of Theorem~5.1~(b) in~\cite{laos}, for it identically applies to the present situation.
\end{proof}
\begin{teor}\label{teor-b continua}
Assume conditions $(c1)$-$(c6)$ on $g$, with $r_0>0$ in $(c5)$ and $(c6)$.
\begin{itemize}
\item[(i)] If $\lambda_P>0$, there exists an ergodic measure $\mu$ such that the dynamical description in Theorem~$\ref{teor-sublineal lin-dis}$ applies for almost every $p$ with respect to $\mu$.
\item[(ii)] If $\alpha_P>0$, then Theorem~$\ref{teor-sublineal lin-dis}$ applies to any $p\in P$, and $b:P\to X^\gamma$ is continuous.
\end{itemize}
\end{teor}
\begin{proof}
For (i) just see Theorem~\ref{teor-principal spectrum} (s3). As for (ii), if $\alpha_P>0$, Theorem~\ref{teor-principal spectrum} (s5.i) justifies the application of Theorem~\ref{teor-sublineal lin-dis} for any $p\in P$. With this dynamics, if $P$ is minimal, the general theory in N\'{u}\~{n}ez et al.~\cite{nuos} for abstract monotone and sublinear skew-product semiflows on $P\times X^\gamma$ over a minimal base flow $P$ implies that $\{(p,b(p))\mid p\in P\}$ is the only strongly positive minimal set and $b$ is continuous.
\par
It remains to complete the proof if $P$ is not minimal. Recall that  $b$ is known to be upper semicontinuous, because it is the nonincreasing limit of a family of continuous super-equilibria: see~\eqref{b(p)} for Neumann or Robin boundary conditions and~\eqref{b(p)-dir} for Dirichlet boundary conditions. So, it suffices to prove that $b$ can also be obtained as the nondecreasing limit of a family of continuous sub-equilibria.
\par
With this purpose, let $z_0\gg 0$ be the vector given in Theorem~\ref{teor-principal spectrum} (s5.ii) related to the uniform persistence. For $z_0\gg 0$ itself there exists a $t_0>0$ such that  $u(t,p,z_0)\geq z_0$ for any $t\geq  t_0$ and $p\in P$. In particular $u(n t_0,p, z_0)\geq z_0$ for $p\in P$ and $n\geq 1$. This means that the constant map $a_0:P\to X^\gamma$, $p\mapsto a_0(p)= z_0$ is a continuous sub-equilibrium for the discrete semiflow generated by the iteration of $\tau_{t_0}$.
Then, we use the method described in the proof of Theorem~3.6 in Novo et al.~\cite{nono2} to build, starting from $a_0$, a nondecreasing family of continuous sub-equilibria $(a_n)_{n}:P\to X^\gamma$ given by $a_n(p)=u(nt_0,p{\cdot}(-nt_0), z_0)$, $p\in P$ and there exists  $\lim_{n\to\infty}a_n(p)=\sup_{n\geq 1}a_n(p)=:s(p)$. In particular, $z_0\leq s(p)\in A(p)$ for $p\in P$, by the pullback attraction of $\A$. The map $s(p)$ constructed in this way is lower semicontinuous.
\par
We affirm that for any $p\in P$, $s(p)=\min\{z\in A(p)\mid u(t,p,z)\geq z_0 \;\forall t\in \R \}$. Let us first check that $u(t,p,s(p))\geq z_0$ for any $t\in \R$. Write $t_n=nt_0$ for $n\geq 1$. If $t\geq 0$, for any $n\geq 1$,
$z_0\leq u(t+t_n,p{\cdot}(-t_n),z_0)= u(t,p,u(t_n,p{\cdot}(-t_n),z_0))\to u(t,p,s(p))$ by the continuity of $u(t,p,\,{\cdot}\,)$ for $t\geq 0$, so that $z_0\leq u(t,p,s(p))$ for $t\geq 0$. Also for $t<0$, for $n$ big enough $t+t_n\geq t_0$ and $z_0\leq u(t+t_n,p{\cdot}(-t_n),z_0)$. Recall that $(p,s(p))$ has a full orbit and let us see that $u(t+t_n,p{\cdot}(-t_n),z_0)\to u(t,p,s(p))$, so that $z_0\leq  u(t,p,s(p))$. Otherwise, taking a subsequence if necessary, $u(t+t_n,p{\cdot}(-t_n),z_0)\to z_1\not= u(t,p,s(p))$, but moving forwards and applying the semicocycle identity and the continuity of $u(-t,p{\cdot}t,\,{\cdot}\,)$, $u(-t,p{\cdot}t,u(t+t_n,p{\cdot}(-t_n),z_0))=u(t_n,p{\cdot}(-t_n),z_0)\to u(-t,p{\cdot}t,z_1)\not= u(-t,p{\cdot}t,u(t,p,s(p)))=s(p)$, which is absurd. It remains to check that $s(p)\leq z$ for any $z\in A(p)$ such that $u(t,p,z)\geq z_0$ for all $t\in \R$. Since the orbits in $\A$ are full, for each $n\geq 1$ we can take $(p{\cdot}(-t_n),z_n)\in \A$ such that $u(t_n,p{\cdot}(-t_n),z_n)=z$. As $z_0\leq z_n$ for $n\geq 1$, by monotonicity $u(t_n,p{\cdot}(-t_n),z_0)\leq u(t_n,p{\cdot}(-t_n),z_n)=z$ and taking limits $s(p)\leq z$. As a consequence, note that the sequence $(nt_0)_n$ used above can be substituted by any other sequence $(t_n)_n\uparrow\infty$.
\par
Now, it is easy to check that $s$ defines an equilibrium for $\tau$: just take $p\in P$, $t\geq 0$ and a sequence $(t_n)_n\uparrow \infty$ such that there exist $\lim_{n\to\infty} u(t_n-t,p{\cdot}(t-t_n),z_0)=s(p)$ and $\lim_{n\to\infty} u(t_n,p{\cdot}(t-t_n),z_0)=s(p{\cdot}t)$. By the semicocycle property and the continuity of $u(t,p,\,{\cdot}\,)$, $u(t_n,p{\cdot}(t-t_n),z_0)=u(t,p,u(t_n-t,p{\cdot}(t-t_n),z_0))\to u(t,p,s(p))$ as $n\to\infty$ and thus, $s(p{\cdot}t)=u(t,p,s(p))$, as wanted.
\par
It only remains to check that $s(p)=b(p)$ for any $p\in P$. Have in mind that, whenever we reduce attention to a minimal set $M\subset P$, due to Theorem~\ref{teor-sublineal lin-dis}, the general theory in~\cite{nuos} applied to the semiflow $\tau$ restricted to $M\times X^\gamma$ is only compatible with the fact that $s(p)=b(p)$.
\par
So, let us fix a $p\in P$ and consider its $\alpha$-limit set  which is a compact invariant set and thus contains a minimal set $M$ of $P$. Take a $p_1\in M$, recall that $s(p_1)=b(p_1)$ and choose a sequence of real numbers $(t_n)_n\uparrow \infty$ such that $p{\cdot}(-t_n)\to p_1$. By taking a subsequence if necessary, we can assume that $(p{\cdot}(-t_n),s(p{\cdot}(-t_n)))_{n\geq 1}\subset \A$ and  $(p{\cdot}(-t_n),b(p{\cdot}(-t_n)))_{ n\geq 1}\subset \A$ converge, respectively to $(p_1,z_1)$ and $(p_1,z_2)$. Since $s(p{\cdot}(-t_n))\leq b(p{\cdot}(-t_n))$ for $n\geq 1$, with the respective semicontinuity properties of $s$ and $b$, we get that $s(p_1)\leq z_1\leq z_2\leq b(p_1)$, so that they all coincide. Then, fixed an $\varepsilon>0$ as small as wanted, for sufficiently big $n$,
\[
(1-\varepsilon)\,b(p{\cdot}(-t_n))\leq s(p{\cdot}(-t_n))\leq b(p{\cdot}(-t_n))\,.
\]
Taking $u(t_n,p{\cdot}(-t_n),\,{\cdot}\, )$, the sublinearity and monotonicity of the semiflow and the fact that both $s$ and $b$ are equilibria for  $\tau$ lead us to
$(1-\varepsilon)\,b(p)\leq s(p)\leq b(p)$, and since $\varepsilon>0$ is arbitrarily small, we can conclude that $s(p) = b(p)$ and the proof is finished.
\end{proof}
%%%%%%%%%%%%%%%%%%%%%%%%%%%%%%%%%%%%%%%%%%%%%%%%%%%%%%%%%%%%%%%%%%%%%%%%%%%%%%%%%%%%%%%%%%%%%%%%%%%%%%%%%%%%%%%%%%%%%%%%%55
%%%%%%%%%%%%%%%%%%%%%%%%%%%%%%%%%%%%%%%%%%%%%%%%%%%%%%%%%%%%%%%%%%%%%%%%%%%%%%%%%%%%%%%%%%%%%%%%%%%%%%%%%%%%%%%%%%%%%%%5
\section{Purely dissipative problems over a compact base flow}\label{sect-purely dissipative}\noindent
In this section we keep the basic notation and terminology from Section~\ref{sect-the problem} and we consider a family~\eqref{pdefamilynl} of purely dissipative parabolic problems over a compact and connected base flow $P$, that is, we assume that $r_0=0$ in condition (c5) so that the dissipative term $g$ is always active. As in the linear-dissipative case, a first goal is to describe the structure of the attractor taking into account  the topological and ergodic properties of $P$. This is done once more through principal spectral theory. Recall that  $\Sigma_{\rm{pr}}=[\alpha_P,\lambda_P]$ is related only to the linear part of the problems.
\par
When $P$ is minimal and uniquely ergodic, and thus $\Sigma_{\rm{pr}}=\{\lambda_P\}$, the structure of the attractor $\A$ has been described, with Neumann or Robin boundary conditions, in Cardoso et al.~\cite{cardoso} for $\lambda_P<0$ (see Proposition~5) and $\lambda_P>0$ (see Theorem~11), and in Caraballo et al.~\cite{calaobsa} for $\lambda_P=0$ (see Proposition~5.12). When $\lambda_P=0$, the upper boundary map $b(p)$ of $\A$ is either identically null so that $\A=P\times\{0\}$, or else it is strongly positive over a set of null measure and of first category of Baire, that is, a small set in both measure and topology senses.
In particular, when the associated 1-dim linear cocycle $c(t,p)$ satisfies that $\sup_{t\in\R} | \ln c(t,p)|<\infty$ for any $p\in P$, it is proved in~\cite{calaobsa} that $\A=P\times\{0\}$. This makes it clear that the condition $\sup_{t\leq 0} c(t,p)<\infty$ is no longer sufficient to guarantee that $b(p)\gg 0$ in the purely dissipative setting. Stronger conditions are needed in this case and to determine them is the second objective in this section.
\par
Some arguments used in the description of $\A$ in Section~\ref{sect-linear dissipative} also apply in the present setting, so that we get some parts of the following theorem for free.
\begin{teor}\label{teor-purely diss 1}
Assume conditions $(c1)$-$(c5)$ on $g$ with $r_0=0$ in $(c5)$. Then:
\begin{itemize}
\item[(i)] If $\lambda_P < 0$, then $\A=P\times \{0\}$ and it is uniformly exponentially stable.
\item[(ii)] If $\lambda_P = 0$, then $P_0=\{p\in P\mid b(p)=0\}$ is a set of complete measure.
\item[(iii)] If  $\alpha_P\leq 0$, then there exists an ergodic measure $\nu$ such that $b(p)=0$ for almost every $p$ with respect to $\nu$.
\item[(iv)]  If $\alpha_P>0$, then the semiflow $\tau$ is uniformly persistent in the interior of the positive cone, $b(p)\gg 0$ for any $p\in P$ and besides, there exists a $z_0\gg 0$ such that $b(p)\geq z_0$ for any $p\in P$.
\end{itemize}
\end{teor}
\begin{proof}
(i) Just see the proof of Theorem~\ref{teor-principal spectrum}~(s1).
\par
(ii) Take any ergodic measure $\nu$ and let us see that $\nu(P_0)=1$. For the 1-dim linear cocycle $c(t,p)$, let $\lambda$ be the associated Lyapunov exponent for $\nu$, that is, $\lim_{t\to\infty} \ln c(t,p)/t=\lim_{t\to-\infty} \ln c(t,p)/t=\lambda$ for almost every $p$ with respect to $\nu$. By Theorem~2.3 in Johnson et al.~\cite{jops},  $\lambda\in \Sigma_{\text{pr}}=[\alpha_P,0]$.  Then, if $\lambda=0$,
arguing as in Remark~\ref{nota-2 purely}.3, let $P_{\rm{r}}$ be the associated set of recurrent points, which  is invariant and  $\nu(P_{\rm{r}})=1$.  We now fix a sequence $(r_n)_n\downarrow 0$ as $n\to\infty$. Under conditions (c1)-(c5) with $r_0=0$  in (c5) for $g$, it is not difficult to  build a sequence of linear-dissipative families sharing the linear part of the purely dissipative family~\eqref{pdefamilynl},
\begin{equation*}%\label{chafee-infante-2}
\left\{\begin{array}{l} \des\frac{\partial y}{\partial t}  =
 \Delta \, y+h(p{\cdot}t,x)\,y+ g_n(y)\,,\quad t>0\,,\;\,x\in U, \;\,\text{for each}\;p\in P, \\[.2cm]
By=0\,,\quad  t>0\,,\;\,x\in \partial U,
\end{array}\right.
\end{equation*}
such that $g_n:\R\to\R$ satisfies conditions $(c1)$-$(c5)$  with $r_n>0$ in $(c5)$, and $g(p,x,y)\leq g_{n+1}(y)\leq g_{n}(y)$ for $p\in P$, $x\in \bar U$, $y\geq 0$ and $n\geq 1$.  Denoting by $b_n(p)$ the corresponding upper boundary map of the attractor for each $n\geq 1$, by a standard argument of comparison of solutions and formula~\eqref{b(p)} we have that $0\leq b(p)\leq b_{n+1}(p)\leq b_{n}(p)$ for  $p\in P$ and $n\geq 1$. Now, take a $p\in P_{\rm{r}}$. If for some $n\geq 1$, $b_n(p)=0$, also $b(p)=0$; else $b_n(p)\gg 0$ for $n\geq 1$ and then we can argue as in the proof of Theorem~4.17~(a) in~\cite{laos} to get that $0\ll b_n(p)\leq r_n$, meaning that $b_n(p)(x)\leq r_n$ for any $x\in\bar U$. Then, taking limits as $n\to\infty$, it must be $b(p)=0$ too. Summing up, $b(p)=0$ for every $p\in P_{\rm{r}}$, that is, $P_{\rm{r}}\subset P_0$. Since  $\nu(P_{\rm{r}})=1$, also $\nu(P_0)=1$, as we wanted to see.
\par
It remains to consider the possibility that $\lambda<0$ provided that $\alpha_P<0$ as well. But this is the easy case, since then $\lim_{t\to -\infty} \ln c(t,p)/t=\lambda<0$ for almost every $p$ with respect to $\nu$, so that $\lim_{t\to -\infty} c(t,p)=\infty$ for all such $p$ and by Remark \ref{nota-1 purely} $b(p) = 0$. Once more, $\nu(P_0)=1$. Therefore, $P_0$ is a set of complete measure.
\par
(iii) Let  $\nu$ be an ergodic measure such that relation~\eqref{medidas erg} holds for $\alpha_P$. Then, we just argue as in the previous item for $\alpha_P=0$ or $\alpha_P<0$.
\par
(iv) Also the proof of Theorem~\ref{teor-principal spectrum}~(s5.ii) is independent of whether the problems are linear-dissipative or purely dissipative. Finally, argue as in the third  paragraph in the proof of Theorem~\ref{teor-b continua} to conclude that  $0\ll z_0\leq b(p)$ for any $p\in P$, for the vector $z_0\gg 0$ linked to the uniform persistence. The proof is complete.
\end{proof}
Before we proceed, we show in the next result that also in the purely dissipative context we can say more when the sublinear condition (c6) is assumed.
\begin{teor}\label{teor-sublineal purely-dis}
Assume conditions $(c1)$-$(c6)$ on $g$, with $r_0=0$ in $(c5)$ and $(c6)$.
\begin{itemize}
\item[(i)] If  $p_0\in P$ is such that $b(p_0)\gg 0$ and $\lim_{t\to\infty} \|b(p_0{\cdot}t)\|=0$, then also $\lim_{t\to \infty} u(t,p_0,z)=0$ for any $z\in X$.
\item[(ii)] If  $p_0\in P$ is such that $b(p_0)\gg 0$ and $\limsup_{t\to\infty} \|b(p_0{\cdot}t)\|>0$, then:
\begin{itemize}
\item[(ii.1)] For any $0< z\leq b(p_0)$ it holds that $\limsup_{t\to \infty} \|u(t,p_0,z)\|>0$ and $\lim_{t\to\infty} b(p_0{\cdot}t)-u(t,p_0,z)=0$.
\item[(ii.2)] If $z\gg 0$ is such that $b(p_0)\leq z$, then $\lim_{t\to\infty} u(t,p_0,z)-b(p_0{\cdot}t)=0$.
\end{itemize}
\item[(iii)] If $\alpha_P>0$, then the dynamical description in {\rm (ii)} applies to any $p_0\in P$, and $b:P\to X^\gamma$ is continuous.
\end{itemize}
\end{teor}
\begin{proof}
(i) Write $-|z|\leq z\leq |z|$ for the map $|z|(x)=|z(x)|$, $x\in\bar U$ and take a $\lambda>1$ such that $|z|\leq \lambda\, b(p_0)$. Then, for $t>0$, by monotonicity and sublinearity of the semiflow, $0\leq u(t,p_0,|z|)\leq u(t,p_0,\lambda\, b(p_0))\leq \lambda\, u(t,p_0, b(p_0))=\lambda\, b(p_0{\cdot}t)\to 0$ as $t\to \infty$. Since we can deduce from condition (c4) that $u(t,p_0,-|z|)=-u(t,p_0,|z|)$, then for $t\geq 0$,   $\|u(t,p_0,z)\|\leq \|u(t,p_0,|z|)\|\to 0$ as $t\to\infty$, and we are done.
\par
(ii) The main ideas  come from those in the proof of Theorem~5.1~(b) in Langa et al.~\cite{laos}, but still there are some differences due to the fact that now $r_0=0$. For the sake of completeness, we give a sketch of the proof.
\par
(ii.1) Let $0< z<b(p_0)$. First of all, by the strong monotonicity we can assume without loss of generality that $z\gg 0$. Then, take a $\lambda>1$ such that $b(p_0)\leq \lambda\, z$. This time, as before, $b(p_0{\cdot}t)\leq \lambda\, u(t,p_0,z)$ for $t\geq 0$ and we get that $l:=\limsup_{t\to \infty} \|u(t,p_0,z)\|>0$. The map
\begin{equation}\label{lambda(t)}
\lambda(t)=\inf\{\lambda\geq 1\mid b(p_0{\cdot}t)\leq \lambda\, u(t,p_0,z)\},\quad t\geq 0
\end{equation}
is nonincreasing and $\lambda(0)>1$. Take $\lambda_0=\lim_{t\to \infty}\lambda(t)\geq 1$. If $\lambda_0=1$ we are done. Argue by contradiction and assume that $\lambda_0>1$. Then, take a sequence $(t_n)_n\uparrow \infty$ such that $\|u(t_n,p_0,z)\|\geq l/2$ for $n\geq 1$. Recalling that semiorbits are bounded and relatively compact, without loss of generality we can assume that $(p_0{\cdot}t_n,u(t_n,p_0,z))\to (p_1,z_1)$ with $\|z_1\|\geq l/2>0$. At this point, we consider the auxiliary family of parabolic PDEs given for  each $p\in P$ by
\begin{equation*}%\label{pde-auxiliar}
\left\{\begin{array}{l} \des\frac{\partial y}{\partial t}  =
 \Delta \, y+h(p{\cdot}t,x)\,y+\lambda_0\,g\Big(p{\cdot}t,x,\frac{y}{\lambda_0}\Big),\quad t>0\,,\;\,x\in U,
  \\[.2cm]
By:=\bar\alpha(x)\,y+\kappa\,\des\frac{\partial y}{\partial n} =0\,,\quad  t>0\,,\;\,x\in \partial U,\,
\end{array}\right.
\end{equation*}
and denote the mild solutions of the associated ACPs by $v(t,p,\wit z)$. With condition (c6) on $g$, by comparison we obtain that $u(t,p,\wit z)\leq v(t,p,\wit z)$ for $p\in P$, $\wit z\geq 0$ and $t\geq 0$, and it is easy to check that $v(t,p, \lambda_0\,\wit z)=\lambda_0\, u(t,p,\wit z)$ for $p\in P$, $\wit z\geq 0$, $t\geq 0$. Then, with the strict sublinearity condition (c6), for the initial condition $\lambda_0\,z_1>0$ we can assert that $u(\varepsilon_1,p_1,\lambda_0\, z_1)\ll v(\varepsilon_1,p_1,\lambda_0\, z_1)=\lambda_0\, u(\varepsilon_1,p_1,z_1)$ for some $\varepsilon_1>0$. By continuity, we can take constants $1<\lambda_1<\lambda_0<\lambda_2$ and an $n_0>0$ so that  $u(\varepsilon_1,p_0{\cdot}t_n,\lambda_2\, u(t_n,p_0,z))\ll \lambda_1 u(\varepsilon_1,p_0{\cdot}t_n, u(t_n,p_0,z))=\lambda_1 u(t_n+\varepsilon_1,p_0,z)$ for $n\geq n_0$. Finally, since $\lambda_0=\lim_{n\to\infty}\lambda(t_n)$ and $\lambda_0<\lambda_2$, there is an $n_1\geq n_0$ such that $b(p_0{\cdot}t_n)\leq \lambda_2\, u(t_n,p_0,z)$ for $n\geq n_1$, by~\eqref{lambda(t)}. But then,
$b(p_0{\cdot}(t_n+\varepsilon_1))\leq u(\varepsilon_1,p_0{\cdot}t_n,\lambda_2\, u(t_n,p_0,z))\ll \lambda_1 u(t_n+\varepsilon_1,p_0,z)$ for $n\geq n_1$, which means that $\lambda(t_n+\varepsilon_1)\leq \lambda_1$ for $n\geq n_1$, but this is absurd since $\lambda_1<\lambda_0$.
\par
(ii.2) Let $z\geq  b(p_0)$. Since the arguments are essentially the same as before, we just note that the appropriate map to be considered is
\begin{equation*}
\lambda(t)=\inf\{\lambda\geq 1\mid u(t,p_0,z) \leq \lambda\, b(p_0{\cdot}t) \},\quad t\geq 0
\end{equation*}
and that this time we take a sequence $(t_n)_n\uparrow \infty$ such that $\|b(p_0{\cdot}t_n)\|\geq l/2$, $n\geq 1$, for $l:=\limsup_{t\to\infty} \|b(p_0{\cdot}t)\|>0$ and $(p_0{\cdot}t_n,b(p_0{\cdot}t_n))\to (p_1,z_1)$ with $z_1>0$.
 \par
(iii) The fact that for any  $p_0\in P$,  $b(p_0)\gg 0$ and $\limsup_{t\to\infty} \|b(p_0{\cdot}t)\|>0$ follows from Theorem~\ref{teor-purely diss 1}~(iv). As a consequence (ii) applies, which is the counterpart of Theorem~\ref{teor-sublineal lin-dis} in this purely dissipative setting and then the same proof as that of Theorem~\ref{teor-b continua}~(ii) permits to assert the continuity of $b$. The proof is finished.
\end{proof}
Since the description is complete if $\lambda_P<0$ or $\alpha_P>0$, hereon we assume that $\alpha_P\leq 0\leq \lambda_P$.
We first concentrate on the search for conditions to delimit when $b(p)\gg 0$ for a fixed $p\in P$.
Example~5.13 in~\cite{calaobsa} of a nontrivial attractor in the purely dissipative framework, with $\Sigma_{\rm{pr}}=\{0\}$, is given for a spatially homogeneous parabolic problem. The construction is based on an ODEs example by Johnson et al.~\cite{joklpa} (p.~79--80), which is again going to be exploited in the following lemma for  ODEs. This lemma is a preliminary result for the case of parabolic PDEs.
\par
Recall that for a single linear equation, the exponential dichotomy corresponds to the usual concept of dichotomy (see~\cite{copp} and~\cite{sase}).
%%%%%%%%%%%%%%%%%%%%%%%%%%%%%%%%%%%%%%%%%%%%%%%%%%%%%%%%%%%%%%%%%%%%%%%%%%%%%%%%%%%%%%%%%%%%%%%%%%%%%%%%%%%%%%%%%%%%%%%%
\begin{lema}\label{lema-odes}
Let us consider the scalar non-autonomous purely dissipative problem on the half plane $w\geq 0$ given by the ODE
\begin{equation}\label{ode}
w'=\frac{1}{\theta-1}\,a_0(t)\,w-\frac{1}{\theta-1}\,w^\theta,
\end{equation}
with $\theta>1$ and $a_0:\R\to\R$ is a bounded and uniformly continuous map.
\begin{itemize}
\item[(i)]
If
\begin{equation}\label{integrable}
e^{\int_0^t \! a_0} \in L^1((-\infty,0])\,,
\end{equation}
then equation \eqref{ode} has an entire positive and bounded solution $w_0(t)$, $t\in \R$.
\item[(ii)] If the continuous spectrum of the linear ODE $v'=a_0(t)\,v$ has the form $[\lambda_1,\lambda_2]$ with $0\leq \lambda_1\leq \lambda_2$  and equation~\eqref{ode} has an entire positive and bounded solution $w_0(t)$, then~\eqref{integrable} holds.
\end{itemize}
\end{lema}
\begin{proof}
(i) Note that the null map is a solution of equation~\eqref{ode}. For $w>0$ let us make the change of variables $v=w^{1-\theta}$ in~\eqref{ode} in order to obtain the scalar nonhomogeneous linear equation
\begin{equation}\label{v+1}
v'= -a_0(t)\,v+1\,.
\end{equation}
Under condition \eqref{integrable}, it is not difficult to check that the formula
\begin{equation}\label{v}
v(t)=\int_{-\infty}^t e^{-\int_s^t \! a_0} \,ds\,,\quad t\in \R
\end{equation}
defines a positive solution of the previous problem so that, undoing the change of variables, $w_0=v^{\frac{1}{1-\theta}}$, $t\in\R$ is a positive entire solution of equation~\eqref{ode}. Finally, note that $a_0(t)$ is  bounded, and in particular there is a $C>0$ such that $a_0(t)\leq C$ for any $t\in\R$. Then, for any $t\in\R$ and any $s\leq t$, $\int_{s}^t a_0\leq C(t-s)$. From here it easily follows  that $\frac{1}{C}\leq v(t)$ for $t\in\R$, so that the map $w_0(t)$ is bounded on $\R$.
\par
(ii) Note that if $\lambda_1>0$, the continuous spectrum of the linear ODE $v'= -a_0(t)\,v$ is $[-\lambda_2,-\lambda_1]$ and $0$ lies in the resolvent set, that is, this equation has an exponential dichotomy with trivial projectors (more precisely, with full stable subspace and null unstable subspace)  and~\eqref{v} gives the expression of the only bounded solution of~\eqref{v+1} (see Coppel~\cite{copp}). In particular, for $t=0$ we get condition~\eqref{integrable}. In fact, the interesting case is $\lambda_1=0$, but  $\lambda_1>0$ has been included for future convenience.
\par
So, assume that $\lambda_1=0$. We use an approximation method for the proof. Let $P$ be the hull of $a_0(t)$ with the usual translation flow $(t,p)\in\R\times P\mapsto p{\cdot}t\in P$. $P$ is compact and connected and the map $a:P\to \R$, $p\mapsto a(p)=p(0)$ is continuous. Take a sequence $(\varepsilon_n)_{n}\downarrow 0$ and for each $n\geq 1$ consider the family of ODEs
\begin{equation}\label{perturbed}
w'=\frac{1}{\theta -1}\,(\varepsilon_n+a(p{\cdot}t))\,w-\frac{1}{\theta -1}\,w^\theta\,,\quad p\in P\,,
\end{equation}
whose solutions generate a skew-product semiflow $\tau_n$ on $P\times \R_+$. Note that the structure of this family falls in what we have called the purely dissipative context and the dissipative term is strictly sublinear for $w>0$. To determine the continuous spectrum of these problems, we look at the linear part
\begin{equation}\label{perturbed lineal}
w'=\frac{1}{\theta -1}\,(\varepsilon_n+a(p{\cdot}t))\,w\,,\quad p\in P\,.
\end{equation}
If $[0,\lambda_2]$ with $0\leq \lambda_2$ is the continuous spectrum of $v'=a_0(t)\,v$, then $[\frac{\varepsilon_n}{\theta-1},\frac{\lambda_2+\varepsilon_n}{\theta-1}]$ is the continuous spectrum for the whole family~\eqref{perturbed lineal}, since the spectrum is preserved along the hull $P$ of $a_0(t)$ and  $\theta>1$.
Here note that the conclusions of Theorems~\ref{teor-purely diss 1} and \ref{teor-sublineal purely-dis} apply to purely dissipative families of ODEs, such as~\eqref{perturbed}. Since $\frac{\varepsilon_n}{\theta-1}>0$,  for each $n\geq 1$ there is a continuous real map $b_n(p)$ which gives the upper boundary map of the attractor for $\tau_n$, such that for any $p\in P$, $b_n(p{\cdot}t)$, $t\in\R$ is a strictly positive and bounded, both above and away from $0$, entire solution  of problem~\eqref{perturbed} for $p$ and $n$, and any other positive solution of~\eqref{perturbed} tends to it as $t\to \infty$.
\par
Besides, the hypothesis that there exists an entire positive and bounded solution $w_0(t)$ of~\eqref{ode} implies that the  upper boundary map $b(p)$ of the attractor for the extended family from~\eqref{ode} to the hull $P$ satisfies that $b(a_0)>0$. By the usual results of comparison of solutions and formula~\eqref{b(p)} in the ODEs case (see Caraballo et al.~\cite{caloNonl}), $b(p)\leq b_n(p)$ for $p\in P$ and $n\geq 1$ and the sequence $(b_n(p))_n$ is nonincreasing. In particular, there exists the limit $\lim_{n\to\infty} b_n(a_0)=b_0= b(a_0)>0$.
\par
After these arguments, let us do the change of variables $v=w^{1-\theta}$ to transform~\eqref{perturbed} for each $n\geq 1$ and for $p=a_0$ into the nonhomogeneous linear equation
\begin{equation}\label{ode-perturbed}
v'= -(\varepsilon_n +a_0(t))\,v+1\,.
\end{equation}
Now the continuous spectrum of the linear ODE $v'= -a_0(t)\,v$ is $[-\lambda_2,0]$, so that the linear problem $v'= -(\varepsilon_n +a_0(t))\,v$ has an exponential dichotomy with trivial projectors and there exists a unique bounded solution for~\eqref{ode-perturbed} given by the formula
\[
v_n(t)=\int_{-\infty}^t e^{-\varepsilon_n(t-s)-\int_s^t \! a_0} \,ds\,,\quad t\in \R\,,
\]
and any other positive solution $v(t)$ of~\eqref{ode-perturbed} converges to $v_n(t)$ as $t\to \infty$ exponentially fast. Note that  necessarily $v_n(t)$ corresponds to $b_n(a_0{\cdot}t)$ for each $n\geq 1$ by the change of variables.
In particular, $v_n(0)=\int_{-\infty}^0 e^{\varepsilon_n t+\int_0^t \!a_0} \,dt$.
Since the continuous maps $f_n(t)=e^{\varepsilon_n t+\int_0^t \! a_0}$ defined on $(-\infty,0]$ satisfy $0\leq f_n\leq f_{n+1}$ for $n\geq 1$,  Lebesgue's monotone convergence theorem says that
\[
\lim_{n\to\infty} v_n(0) = \int_{-\infty}^0 e^{ \int_0^t \! a_0} \,dt =:\beta\in (0,\infty]\,.
\]
Hence, undoing the change of variables,  $w_n(0)=v_n(0)^{\frac{1}{1-\theta}}$ is the initial value $b_n(a_0)$ of the entire solution $b_n(a_0{\cdot}t)$ and, as seen before, $\lim_{n\to\infty} b_n(a_0)=b_0>0$. Since
\[
\lim_{n\to\infty}w_n(0)=\beta^{\frac{1}{1-\theta}}=b_0>0\Longleftrightarrow \beta\in (0,\infty) \Longleftrightarrow e^{ \int_0^t \! a_0}\in L^1((-\infty,0]) \,,
\]
the proof is finished.
\end{proof}
\begin{notas}\label{nota-odes}
1. The uniform continuity of $a_0(t)$ is required in order to build its hull. Note that the result might have been given with a collective formulation for a family of equations over a general compact and connected flow $P$.
\par
2. It is easy to check that condition~\eqref{integrable} implies that the Lyapunov exponent $\lambda_s^-(a_0)\geq 0$, so that  necessarily $\lambda_2\geq 0$ if it holds. The restriction in (ii) is $\lambda_1\geq 0$.
\par
3. Condition~\eqref{integrable} holds for the almost periodic maps $a_0(t)$ with associated  spectrum $\Sigma=\{0\}$ considered in Example~5.13 in~\cite{calaobsa}. More precisely, whenever $\widetilde a:\R\to\R$ is an almost periodic  function with zero mean value and whose integral $\int_0^t \!\widetilde  a$ grows like (or faster than) $t^\beta$ as $t\to \infty$ for some $0<\beta<1$, the almost periodic map with zero mean value $a_0(t)=\widetilde a(-t)$, $t\in\R$ satisfies the former condition (see Johnson et al.~\cite{joklpa}). Such maps $\widetilde a$ have been explicitly built in the literature by several authors, such as Poincar\'{e}~\cite{poin}, Conley and Miller~\cite{comi}, Zhikov and Levitan~\cite{zhle} and Johnson and Moser~\cite{jomo}.
\end{notas}
The next theorem is fundamental in the paper, due to its later application to much more general problems in Theorem \ref{teor-coro}. We determine the  integrability condition~\eqref{a1} on $c(t,p_0)^{\theta-1}$ for $t\leq 0$, sufficient to guarantee a nontrivial global attractor for a family of purely dissipative problems including Chaffe-Infante equations (for $\theta=3$). Note that $\theta$ is the exponent in the nonlinear term. Condition~\eqref{a1} is also necessary under some restrictions. This result is specially relevant for  $\Sigma_{\text{pr}}=\{0\}$, stating an if and only if condition for $b(p_0)\gg 0$, at least with Neumann or Robin boundary conditions. For simplicity, we reduce attention to the positive cone (see Proposition~\ref{prop-atractor cono positivo}), but note that the dissipative term can be extended to the full line in an odd way,  as in condition (c4).
\begin{teor}\label{teor-atractor no trivial}
Let us consider the family of purely dissipative problems over a continuous flow on a compact and connected metric space $P$, given for $y\geq 0$ by
\begin{equation}\label{chafee-infante}
\left\{\begin{array}{l} \des\frac{\partial y}{\partial t}  =
 \Delta \, y+h(p{\cdot}t,x)\,y-\rho\,y^\theta\,,\quad t>0\,,\;\,x\in U, \;\,\text{for each}\;p\in P, \\[.2cm]
By:=\bar\alpha(x)\,y+\kappa\,\des\frac{\partial y}{\partial n} =0\,,\quad  t>0\,,\;\,x\in \partial U,
\end{array}\right.
\end{equation}
with Neumann, Robin or Dirichlet boundary conditions, where $h\in C(P\times \bar U)$, $\rho>0$ and $\theta>1$, and let $\A_+$ be the global attractor.  Associated to the linear problems~\eqref{linealizada}, let $e(p)\gg 0$, $p\in P$ be the normalized vectors leading the principal bundle in the continuous separation, let $c(t,p)$ be the associated continuous 1-dim linear cocycle and denote by $\Sigma_{\text{pr}}=[\alpha_P,\lambda_P]$ the principal spectrum. Then:
\begin{itemize}
\item[(i)] Whenever for some $p_0\in P$ the following assumption holds,
\begin{equation}\label{a1}
c(t,p_0)^{\theta-1}\in L^1((-\infty,0])\,,
\end{equation}
the  section of the attractor $A_+(p_0)\not=\{0\}$, or equivalently $b(p_0)\gg 0$.
\item[(ii)] If $\Sigma_{\text{pr}}=[\alpha_P,\lambda_P]$ with $0\leq \alpha_P\leq \lambda_P$ and, with either Neumann or Robin boundary conditions, for a certain $p_0\in P$ it is $b(p_0)\gg 0$, then condition {\rm \eqref{a1}} holds.
\end{itemize}
\end{teor}
\begin{proof}
It is well-known that, with $P$ compact, any continuous cocycle is cohomologous to a smooth one (see Lemma~3.2 in Johnson et al.~\cite{jops}). More precisely, we assert that there exists a smooth cocycle $c_0(t,p)$ which we write for convenience as
\[
c_0(t,p)=e^{\int_0^t \frac{1}{\theta-1}\,a(p{\cdot}s)\,ds}
\]
for some $a\in C(P)$, and a continuous map $f:P\to \R\setminus\{0\}$ such that $c_0(t,p)= f(p{\cdot}t)\,c(t,p)\,f(p)^{-1}$ for $p\in P$ and $t\in\R$. Besides, since $P$ is compact and connected, either $0<c_1\leq f(p)\leq c_2$ for $p\in P$ or $-c_2\leq f(p)\leq -c_1$ for $p\in P$, for some constants $0<c_1<c_2$. We can assume without loss of generality that $f$ takes positive values.
Let us fix a  $p_0\in P$ and write
\[
c_0(t,p_0)^{\theta-1}=e^{\int_0^t a(p_0{\cdot}s)ds}=f(p_0{\cdot}t)^{\theta-1}\,c(t,p_0)^{\theta-1}\,f(p_0)^{1-\theta}\,,\quad t\in\R\,.
\]
Thus, thanks to the boundedness of $f$ over $P$, condition~\eqref{a1} holds if and only if condition~\eqref{integrable} holds for the map $a(p_0{\cdot}t)$, $t\in \R$.
\par
(i) Let $p_0\in P$ be such that~\eqref{a1} holds. Then,  as a consequence of the previous argumentation and according to Lemma~\ref{lema-odes}~(i), there exists  an entire positive and bounded solution $w_0(t)$, $t\in\R$ of the scalar ODE given for $p=p_0$ within the family of purely dissipative and sublinear ODEs on the positive half plane $w\geq 0$,
\begin{equation}\label{ode-w}
w'=\frac{1}{\theta-1}\,a(p{\cdot}t)\,w-\frac{1}{\theta-1}\,w^\theta,\quad {p\in P}.
\end{equation}
We use the notation $w(t,p,r)$ for the value at time $t$ of the solution of the scalar ODE for $p$ with initial condition $r\geq 0$. It is well-known (see Caraballo et al.~\cite{caloNonl}) that this family of equations admits a global attractor which we write as  $\cup_{p\in P} \{p\}\times [0,b^*(p)]\subset P\times \R_+$, and for $r>0$ big enough,
\begin{equation}\label{b odes}
b^*(p)=\lim_{t\to \infty} w(t,p{\cdot}(-t),r)\,,\quad p\in P\,.
\end{equation}
Note that the existence of $w_0(t)$ implies that $b^*(p_0)>0$.
\par
Now, the 1-dim cocycle associated to the linear family for $p\in P$,
\begin{equation}\label{lineal modificada}
\left\{\begin{array}{l} \des\frac{\partial y}{\partial t}  =
 \Delta \, y+\Big(h(p{\cdot}t,x)-\frac{1}{\theta-1}\,a(p{\cdot}t)\Big)\,y\,,\quad t>0\,,\;\,x\in U, \\[.2cm]
By:=\bar\alpha(x)\,y+\kappa\,\des\frac{\partial y}{\partial n} =0\,,\quad  t>0\,,\;\,x\in \partial U,
\end{array}\right.
\end{equation}
is given by $\wit c(t,p)=c(t,p)\,e^{-\int_0^t \frac{1}{\theta-1}\,a(p{\cdot}s)\,ds}=c(t,p)\,c_0(t,p)^{-1}=f(p)\,f(p{\cdot}t)^{-1}$ and $\eta_1\leq \wit c(t,p)\leq \eta_2$ for certain $\eta_1,\eta_2>0$, for any $p\in P$ and $t\in \R$. In fact this family has the same principal bundle as the one for  the linear coefficient $h$ (see~\cite{calaobsa}). Then, denoting by $\wit\phi(t,p)$ the linear cocycle providing the solutions of \eqref{lineal modificada}, $z(t,p)=\wit \phi(t,p)\,e(p)=\wit c(t,p)\,e(p{\cdot}t)$, $t\geq 0$ is the solution of the previous linear problem for $p$ starting at $e(p)$ and  $z(s,p):=\wit c(s,p)\,e(p{\cdot}s)$ for $s\leq 0$ defines a backward extension, since for each $s\leq 0$ and $0\leq t\leq -s$, $\wit\phi(t,p{\cdot}s)\,z(s,p)=\wit c(s,p)\,\wit\phi(t,p{\cdot}s)\,e(p{\cdot}s)=\wit c(s,p)\,\wit c(t,p{\cdot}s)\,e(p{\cdot}(s+t))=\wit c(s+t,p)\,e(p{\cdot}(s+t))=z(s+t,p)$.
\par
Back to the fixed $p_0$, for each  $s\geq 0$, and $r>0$ big enough, we consider the map
\begin{align*}
y_s(t,x)=&w(t,p_0{\cdot}(-s),r)\,\wit c(t-s,p_0)\, e(p_0{\cdot}(t-s))(x)\\
=& w(t,p_0{\cdot}(-s),r)\,\wit c(-s,p_0)\,z(t,p_0{\cdot}(-s))(x)\,,\quad t\geq 0\,,\;\, x\in\bar U.
\end{align*}
Let us assume for the moment that $h(p_0{\cdot}t,x)$ and $a(p_0{\cdot}t)$ are smooth enough with respect to $t$ and $x$, so that mild solutions of the linear problem~\eqref{lineal modificada} for $p_0{\cdot}(-s)$ become classical solutions (in fact a H\"{o}lder-continuity condition is enough, see Friedman~\cite{frie}).
If this is the case, we can write:
\begin{equation*}
\left\{\begin{array}{l} \des\frac{\partial y_s}{\partial t}  = \Delta \, y_s+h(p_0{\cdot}(t-s),x)\,y_s-\frac{w(t,p_0{\cdot}(-s),r)^{\theta}\,\wit c(t-s,p_0)\,e(p_0{\cdot}(t-s))(x)}{\theta-1} ,\\[.1cm] \hspace{9,5cm} \quad t>0\,,\;\,x\in U, \\
y_s(0,x)=r\,\wit c(-s,p_0)\, e(p_0{\cdot}(-s))(x)\,,\quad x\in \bar U,
\\[.1cm]
By_s:=\bar\alpha(x)\,y_s+\kappa\,\des\frac{\partial y_s}{\partial n} =0\,,\quad  t>0\,,\;\,x\in \partial U.
\end{array}\right.
\end{equation*}
Note that for any $s\geq 0$, for  $t\geq 0$ and $x\in \bar U$ we can bound
\[
-\frac{w(t,p_0{\cdot}(-s),r)^{\theta}\,\wit c(t-s,p_0)\,e(p_0{\cdot}(t-s))(x)}{\theta-1} \leq -\frac{1}{(\theta-1)\,
\eta_2^{\theta-1}}\,y_s(t,x)^{\theta}
\]
for $\eta_2>0$ the upper bound for the 1-dim cocycle $\wit c$ and $1$ the upper bound for $e(p)^{\theta-1}(x)$ for $p\in P$, $x\in\bar U$. Then, for each $s\geq 0$ fixed, $y_s(t,x)$ is a regular subsolution of
the problem for $p=p_0{\cdot}(-s)$ within the family on the positive cone
\begin{equation}\label{casi el mismo}
\left\{\begin{array}{l} \des\frac{\partial y}{\partial t}  =
 \Delta \, y+h(p{\cdot}t,x)\,y-\frac{1}{(\theta-1)\, \eta_2^{\theta-1}}\,y^{\theta}\,,\quad t>0\,,\;\,x\in U,  \\[.2cm]
By=0\,,\quad  t>0\,,\;\,x\in \partial U,
\end{array}\right.
\end{equation}
and with the regularity assumptions on $h(p_0{\cdot}t,x)$,  at least for the problems along the orbit of $p_0$, mild solutions are classical.
Then, denoting by $\bar\tau(t,p,z)=(p{\cdot}t,\bar u(t,p,z))$ the induced skew-product semiflow by these last problems, for $p=p_0{\cdot}(-s)$ we can conclude, using well-known comparison principles of classical solutions (for instance, see Theorem~1 in Fife and Tang~\cite{fita} in a larger context), that
\[
y_s(t,x)\leq \bar u(t,p_0{\cdot}(-s),r\,\wit c(-s,p_0) \, e(p_0{\cdot}(-s)))(x)\,, \quad t\geq 0\,, \; x\in \bar U.
\]
Now, fixed $e_0\gg 0$ the one in \eqref{bvp}, since $\wit c(t,p_0)$ is bounded above for $t\in\R$, we have that $r\,\wit c(-s,p_0) \, e(p_0{\cdot}(-s))\leq \bar r e_0$ for $s\geq 0$, for a large enough $\bar r$ so that~\eqref{b(p)} applies for $\bar\tau$, and then $y_s(t,x)\leq \bar u(t,p_0{\cdot}(-s),\bar r e_0)(x)$ for  $s, t\geq 0$, $x\in\bar U$. Taking $s=t$ we get that $y_t(t,x)=w(t,p_0{\cdot}(-t),r)\, e(p_0)(x)\leq \bar u(t,p_0{\cdot}(-t),\bar r e_0)(x)$ for any $t\geq 0$, $x\in\bar U$; that is, $w(t,p_0{\cdot}(-t),r)\,e(p_0)\leq \bar u(t,p_0{\cdot}(-t),\bar r e_0)$ for  $t\geq 0$. At this point we apply~\eqref{b odes} and~\eqref{b(p)}  to deduce, taking limits as $t\to\infty$, that $b^*(p_0)\, e(p_0)\leq \bar b(p_0)$, for $\bar b$ the upper boundary map of the global attractor for $\bar\tau$. Thus, $\bar b(p_0)\gg 0$ and the global attractor for problems~\eqref{casi el mismo} is nontrivial.
\par
Note that the only difference between this  family of problems and the family~\eqref{chafee-infante} in the statement  is the constant appearing in the dissipative term. But this is unimportant, since it is easy to check that if $y(t,x)$ is a strongly positive entire bounded solution of~\eqref{casi el mismo} for $p=p_0$, then $\beta\,y(t,x)$ is a solution of~\eqref{chafee-infante} for $p=p_0$ just by taking the appropriate value of $\beta>0$. Therefore, $b(p_0)\gg 0$, that is, there is a  nontrivial section at $p_0$ in the attractor, as we wanted to prove.
\par
It only remains to withdraw the hypothesis that $h(p_0{\cdot}t,x)$ and $a(p_0{\cdot}t)$ have further regularity properties with respect to $t$ and $x$. If just continuity is assumed, we can  approximate the maps $h(p_0{\cdot}t,x)$ and $a(p_0{\cdot}t)$ by respective sequences $(h_{n})_n,\,(a_n)_n$ of sufficiently regular maps; more precisely, $h_n:\R\times \bar U\to \R$ of class $C^1$ in both arguments and $a_n:\R\to\R$ of class $C^{1\!}$ in $t$ so that  $h_n(t,x)\to h(p_0{\cdot}t,x)$ and $a_n(t)\to a(p_0{\cdot}t)$ as $n\to\infty$ uniformly on compact sets.  Now, for each  $s\geq 0$ and $n\geq 1$, let $w_{s,n}(t,r)$ denote the value at time $t$ of the solution of the shifted by $-s$ scalar ODE
\begin{equation*}
w'=\frac{1}{\theta-1}\,a_n(t-s)\,w-\frac{1}{\theta-1}\,w^\theta
\end{equation*}
with initial condition $r> 0$. Under the former conditions it is well-known that $w_{s,n}(t,r)\to w(t,p_0{\cdot}(-s),r)$ as $n\to\infty$, uniformly for $t$ on compact sets $[0,t_0]$ ($t_0>0$). Analogously, let $z_{s,n}(t,x)$ denote the classical solution of the shifted by $-s$ scalar linear regular PDEs problem
\begin{equation*}
\left\{\begin{array}{l} \des\frac{\partial y}{\partial t}  =
 \Delta \, y+\Big(h_n(t-s,x)-\frac{1}{\theta-1}\,a_n(t-s)\Big)\,y\,,\quad t>0\,,\;\,x\in U, \\[.2cm]
 y(0,x)=e(p_0{\cdot}(-s))(x) \,,\quad x\in\bar U, \\[.1cm]
By =0\,,\quad  t>0\,,\;\,x\in \partial U,
\end{array}\right.
\end{equation*}
Once more, $z_{s,n}(t,x)\to z(t,p_0{\cdot}(-s))(x)$ uniformly on compact sets of  $[0,\infty)\times \bar U$.
Now, since $z(t,p)=\wit c(t,p)\,e(p{\cdot}t)$, $t\geq 0$, we have that $\sup\{z(t,p)(x)\mid t\geq 0, x\in \bar U, p\in P\}\leq \eta_2$. Let us fix an $\eta_2'>\eta_2$. Then, for a fixed $s\geq 0$, given a $t_0>0$ there is an $n_0=n_0(t_0,s)$ such that $z_{s,n}(t,x)\leq \eta_2'$ for $t\in [0,t_0]$, $x\in\bar U$, for any $n\geq n_0$. Then, for $n\geq n_0$ it is easy to check that the regular map
\[
y_{s,n}(t,x)=w_{s,n}(t,r)\,\wit c(-s,p_0)\,z_{s,n}(t,x)\,,\quad t\geq 0\,,\;\, x\in\bar U
\]
is a subsolution  in $[0,t_0]\times \bar U$ of the regular problem
\begin{equation*}
\left\{\begin{array}{l} \des\frac{\partial y}{\partial t}  =
 \Delta \, y+h_n(t-s,x)\,y-\frac{1}{(\theta-1)\, \eta_2^{\theta-1}\,{(\eta_2')}^{\theta-1}}\,y^{\theta}\,,\quad t>0\,,\;\,x\in U,  \\[.2cm]
By=0\,,\quad  t>0\,,\;\,x\in \partial U,
\end{array}\right.
\end{equation*}
so that again the standard comparison principles imply that
\[
y_{s,n}(t,x)\leq \wit y_{s,n}(t,x) \,, \quad t\in[0,t_0]\,, \; x\in \bar U,
\]
where $\wit y_{s,n}(t,x)$ is the classical solution of the previous regular problem with initial value the map $r\,\wit c(-s,p_0) \, e(p_0{\cdot}(-s))(x)$, $x\in \bar U$. Then, taking limits as $n\to \infty$,
\[
y_s(t,x)\leq \wit u(t,p_0{\cdot}(-s),r\,\wit c(-s,p_0) \, e(p_0{\cdot}(-s)))(x)\,, \quad t\in[0,t_0]\,, \; x\in \bar U
\]
for $\wit u$ the nonlinear cocycle associated to the family for $p\in P$ on the positive cone
\begin{equation*}%\label{casi el mismo}
\left\{\begin{array}{l} \des\frac{\partial y}{\partial t}  =
 \Delta \, y+h(p{\cdot}t,x)\,y-\frac{1}{(\theta-1)\, \eta_2^{\theta-1}\,{(\eta_2')}^{\theta-1}}\,y^{\theta}\,,\quad t>0\,,\;\,x\in U,  \\[.2cm]
By=0\,,\quad  t>0\,,\;\,x\in \partial U.
\end{array}\right.
\end{equation*}
Since $t_0$ can be any, we can conclude that for any $s\geq 0$,
\[
y_s(t,x)\leq \wit u(t,p_0{\cdot}(-s),r\,\wit c(-s,p_0) \, e(p_0{\cdot}(-s)))(x)\,, \quad t\geq 0\,, \; x\in \bar U,
\]
and the proof is finished just as before.
\par
(ii) We maintain all the former notation in the proof. This time we only write the proof assuming enough regularity on the coefficients so that mild solutions are classical solutions, but note that in the general case it is necessary to return to the approximation arguments explained in (i).
\par
Let $p_0\in P$ be such that $b(p_0)\gg 0$. First of all, by the arguments in the initial paragraph,  $\Sigma_{\text{pr}}=[\alpha_P,\lambda_P]$ with $0\leq \alpha_P\leq \lambda_P$ is precisely the continuous spectrum of the scalar family of ODEs $v'=\frac{1}{\theta-1}\,a(p{\cdot}t)\,v$, $p\in P$. Thus, that of the family $v'=a(p{\cdot}t)\,v$, $p\in P$ is $[(\theta-1)\alpha_P,(\theta-1)\lambda_P]$  because $\theta>1$.
Besides, by the classical spectral theory of Sacker and Sell~\cite{sase}, the continuous spectrum of the equation $v'=a(p_0{\cdot}t)\,v$ is a compact interval $[\lambda_1,\lambda_2]\subseteq [(\theta-1)\alpha_P,(\theta-1)\lambda_P]$, so that $0\leq \lambda_1\leq\lambda_2$. So, it suffices to check that equation~\eqref{ode-w} for $p_0$ has an entire positive and bounded solution $w_0(t)$, in order to apply Lemma~\ref{lema-odes}~(ii) to get that~\eqref{integrable} holds for the map $a(p_0{\cdot}t)$, $t\in \R$, which is equivalent to condition~\eqref{a1}.
\par
To do so, this time for any $s\geq 0$, for  $t\geq 0$ and $x\in \bar U$ we can write
\[
-\frac{w(t,p_0{\cdot}(-s),r)^{\theta}\,\wit c(t-s,p_0)\,e(p_0{\cdot}(t-s))(x)}{\theta-1}\geq - \frac{\eta}{\theta-1}\,y_s(t,x)^{\theta}
\]
for an appropriate $\eta>0$. Note that here the lower bound $0<\eta_1\leq \wit c(t,p)$ has been used, together with the fact that with Neumann or Robin boundary conditions there exists a $c_0>0$ such that $c_0\leq e(p)(x)$ for any $p\in P$ and $x\in\bar U$ (however, this fails to be true in the Dirichlet case). Then, for each $s\geq 0$, this time $y_s(t,x)$ is a regular supersolution for the problem for $p=p_0{\cdot}(-s)$ in the positive cone  within the family for $p\in P$,
\begin{equation*}
\left\{\begin{array}{l} \des\frac{\partial y}{\partial t}  =
 \Delta \, y+h(p{\cdot}t,x)\,y-   \frac{\eta}{\theta-1}\,y^{\theta}\,,\quad t>0\,,\;\,x\in U,   \\[.2cm]
By=0\,,\quad  t>0\,,\;\,x\in \partial U.
\end{array}\right.
\end{equation*}
Denoting by $\bar\tau(t,p,z)=(p{\cdot}t,\bar u(t,p,z))$  the induced skew-product semiflow for these last problems, we can conclude by the standard comparison principles that
\[
y_s(t,x)\geq \bar u(t,p_0{\cdot}(-s),r\,\wit c(-s,p_0) \, e(p_0{\cdot}(-s)))(x)\,, \quad t\geq 0\,, \; x\in \bar U.
\]
Taking $s=t$ we get that $w(t,p_0{\cdot}(-t),r)\geq y_t(t,x)=w(t,p_0{\cdot}(-t),r)\, e(p_0)(x)\geq \bar u(t,p_0{\cdot}(-t),\bar r e_0)(x)$ for any $t\geq 0$, $x\in\bar U$, for an appropriate $\bar r$ which can be taken as big as necessary so that the pullback formula~\eqref{b(p)} holds, provided that $r$ is big too. That is, $w(t,p_0{\cdot}(-t),r)\geq \bar u(t,p_0{\cdot}(-t),\bar r e_0)(x)$ for  $t\geq 0$, $x\in\bar U$. Taking limits as $t\to \infty$, $b^*(p_0)\geq \bar b (p_0)(x)$, $x\in\bar U$ for the upper boundary  map $\bar b$ of the attractor for $\bar \tau$. Since the hypothesis $b(p_0)\gg 0$ for the family~\eqref{chafee-infante} is independent of the positive constant in the dissipative term, it is clear that $\bar b (p_0)\gg 0$. Then, $b^*(p_0)>0$ and  $b^*(p_0{\cdot}t)$, $t\in\R$ provides us with the desired entire positive and bounded solution of equation~\eqref{ode-w} for $p_0$. The proof is complete.
\end{proof}
For further insight into condition \eqref{a1}, we prove that it presumes an asymptotic behaviour of the 1-dim cocycle as $t\to -\infty$.
\begin{prop}\label{prop-puntos asintoticos}
Assume that $c(t,p)$ is the 1-dim linear cocycle associated to a linear family~\eqref{linealizada} over a compact base flow $P$. If for some $\beta>0$ and for some $p_0\in P$, $c(t,p_0)^{\beta}\in L^1((-\infty,0])$, then $p_0$ is an asymptotic point at $-\infty$, i.e., $\lim_{t\to -\infty} c(t,p_0)=0$, and $c(t,p_0)^{\beta'}\in L^1((-\infty,0])$ for any $\beta'>\beta$.
\end{prop}
\begin{proof}
As it has already been mentioned in the proof of Theorem~\ref{teor-atractor no trivial}, every continuous linear cocycle is cohomologous to a smooth one. By the relation between them, it is enough to prove the result assuming that $c(t,p)$ is a smooth cocycle, that is, $c(t,p)=e^{\int_0^t a(p{\cdot}s)ds}$, $p\in P$, $t\in\R$ for some map $a\in C(P)$. We argue by contradiction and assume that $c(t,p_0)$ does not tend to $0$ as $t\to -\infty$. Then, there exists an $0<\varepsilon_0<1$ and a sequence $(t_n)_n\downarrow -\infty$ such that $c(t_n,p_0)\geq \varepsilon_0$, $n\geq 1$, so that $\ln c(t_n,p_0)\geq \ln\varepsilon_0$. Now, the map $\ln c(t,p_0)=\int_0^t a(p_0{\cdot}s)\,ds$ is uniformly continuous on $t\in \R$ due to its bounded derivative, and then, we can take a $\delta>0$ such that $\ln c(t,p_0)\geq 2\ln\varepsilon_0$ for $t\in [t_n,t_n+\delta]$ for any $n\geq 1$. That is, $c(t,p_0)^\beta\geq \varepsilon_0^{2\beta}$ for $t\in [t_n,t_n+\delta]$ for any $n\geq 1$. But this is absurd, since then $\int_{-\infty}^0 c(t,p_0)^\beta\,dt\geq \sum_{n\geq 1} \delta\,\varepsilon_0^{2\beta}=\infty$. Consequently, $\lim_{t\to -\infty} c(t,p_0)=0$.
Finally, if $\beta'>\beta$, fixed an $\varepsilon>0$ there exists a $t_0<0$ such that for $t\leq t_0$ we can write $c(t,p_0)^{\beta'}=c(t,p_0)^{\beta'-\beta}\,c(t,p_0)^\beta\leq \varepsilon^{\beta'-\beta}\,c(t,p_0)^\beta\in L^{1}((-\infty,0])$. The proof is finished.
\end{proof}
\begin{notas}\label{nota-pullback exponent}
1.  We fall into the framework of Theorem~\ref{teor-atractor no trivial} whenever we start with a single non-autonomous parabolic problem given for $y\geq 0$ by
\begin{equation*}%\label{pdefamily}
\left\{\begin{array}{l} \des\frac{\partial y}{\partial t}  =
 \Delta \, y+h_0(t,x)\,y-\rho\,y^\theta\,,\quad t>0\,,\;\,x\in U, \\[.2cm]
By:=\bar\alpha(x)\,y+\kappa\,\des\frac{\partial y}{\partial n} =0\,,\quad  t>0\,,\;\,x\in \partial U,
\end{array}\right.
\end{equation*}
with $h_0:\R\times\bar U\to \R$ bounded and uniformly continuous, $\rho>0$ and $\theta>1$. Then, one can build the hull $P$  of $h_0$ and define $p{\cdot}t$ as the shift flow on $P$, which is compact and connected, and consider the family of problems over the hull~\eqref{chafee-infante}, where $h:P\times\bar U\to \R$ is defined by $h(p,x)=p(0,x)$ for $p\in P$, $x\in \bar U$, in such a way that we recover the initial problem by taking the element $p(t,x)=h_0(t,x)$.
\par
In particular, almost periodic problems provide us with examples of different dynamical behaviours. Let $h_0(t)$ be an almost periodic function with zero mean value and let $P$ be its hull. If $e^{(\theta-1)\int_0^t h_0(s)ds}\in L^1((-\infty,0])$  we can assure that the family~\eqref{chafee-infante}   with linear coefficient $h(p{\cdot}t)+\gamma_0$ ($\gamma_0$ is the first eigenvalue of the BVP~\eqref{bvp}) has a nontrivial attractor. Note that  $c(t,p)=e^{\int_0^t h(p{\cdot}s)ds}$ is the associated 1-dim cocycle and, by Proposition~\ref{prop-puntos asintoticos}, it satisfies $\lim_{t\to -\infty}c(t,h_0)=0$.  On the other extreme, $h_0$ might be such that  every point $p\in P$ is recurrent at $\pm\infty$: see Example 4.20 in~\cite{laos} for more details. Since the recurrence property is incompatible with the asymptotic character of the cocycle at $-\infty$, at least with Neumann and Robin boundary conditions, we can deduce from Theorem~\ref{teor-atractor no trivial} that $\A=P\times\{0\}$ in those cases.
\par
 2. If for some $p_0$ the pullback Lyapunov exponent $\lambda_{i}'(p_0,z)>0$ for some $z\gg 0$ (equivalently, for any $z\gg 0$), then~\eqref{a1} holds. To see it, apply Proposition~\ref{prop-exponentes Lyap} to get that $0<\lambda_{i}'(p_0,z)=\liminf_{t\to -\infty} \ln c(t,p_0)/t$. Therefore, given $0<\lambda<\lambda_{i}'(p_0,z)$, there exists a $t_0<0$ ($t_0=t_0(p_0,\lambda)$) such that $\ln c(t,p_0)/t> \lambda$ for all $t\leq t_0$. From this, $c(t,p_0)< e^{\lambda t}$ for all $t\leq t_0$, so that the integrability condition in~\eqref{a1} is fulfilled for any $\theta>1$. The  advantage is that this exponent might be numerically computed, since it depends on the behaviour of the positive solutions of the associated linear problems. But note that $\lambda_{i}'(p_0,z)>0$ presumes that $\lambda_P>0$, whereas~\eqref{a1} is compatible with $\lambda_P=0$.
\end{notas}
Our next goal is to extend the application of (i) and (ii) in Theorem~\ref{teor-atractor no trivial} to a class as big as possible of families~\eqref{pdefamilynl}. Theorem~\ref{teor-coro} states to what extent this can be done.  Summing up, the sufficient condition~\eqref{a1} $c(t,p_0)^{\theta-1}\in L^1((-\infty,0])$ for a nontrivial section $A(p_0)$ is applicable whenever  some  additional regularity is assumed on the nonlinear term $g$, the appropriate value of $\theta$ being linked to the regularity of $g$. We first give an easy but still fundamental result for our purposes. The spaces of functions here used have been introduced in Section~\ref{sect-the problem}.
\begin{prop}\label{prop-analisis 1}
{\rm 1.} Assume that $g(p,x,y)\in C^{0,0,n+\beta}(P\times \bar U\times [0,\delta])$ for an integer $n\geq 1$, a $\beta\in (0,1^-]$ and a $\delta>0$, and
\begin{equation}\label{derivadas}
\des\frac{\partial^k g}{\partial y^k}(p,x,0)=0\quad\text{for any}\;\, p\in  P\;\text{and}\;\,  x\in \bar U,\;\,\text{for}\;\, 0\leq k\leq n\,.
\end{equation}
Then, there exists a $\rho_0>0$ such that
\begin{equation}\label{holder}
g(p,x,y)\geq -\rho_0\,y^{n+\beta}\quad \text{for}\;\, p\in P\,,\;x\in \bar U,\; y\in [0,\delta]\,.
\end{equation}
{\rm 2.} If  $g\in C^{0,0,n+1}(P\times \bar U\times [0,\delta])$ satisfies~\eqref{derivadas} and
\begin{equation}\label{n+1}
\des\frac{\partial^{n+1} g}{\partial y^{n+1}}(p,x,0)\leq C<0\quad\text{for any}\;\, p\in  P\;\text{and}\;\,  x\in \bar U,
\end{equation}
then there exist  a $0<\delta_1\leq \delta$ and a  $\wit \rho_0>0$ such that
\begin{equation}\label{holder-2}
g(p,x,y)\leq -\wit \rho_0\,y^{n+1}\quad \text{for}\;\, p\in P\,,\;x\in \bar U,\; y\in [0,\delta_1]\,.
\end{equation}
\end{prop}
\begin{proof}
The proof relies on Taylor's theorem. In the first situation, looking at \eqref{derivadas}, for any $p\in P$, $x\in \bar U$ and $y\in [0,\delta]$ there is an  $\omega(p,x,y)\in (0,1)$ such that
\[
|g(p,x,y)|=\Big| \frac{1}{n!}\,\frac{\partial^ng}{\partial y^n}(p,x,\omega(p,x,y)y)\,y^n - \frac{1}{n!}\,\frac{\partial^ng}{\partial y^n}(p,x,0)\,y^n\Big| \leq \frac{1}{n!}\, \eta \,y^{n+\beta}
\]
for $\eta>0$ given by the $\beta$-H\"{o}lder/Lipschitz continuity of the  $n^{\rm{th}}$-order partial derivative. Then, it suffices to take $\rho_0=\eta/n!$.
\par
In the second situation, we can take a $0<\delta_1\leq \delta$ such that
$\frac{\partial^{n+1} g}{\partial y^{n+1}}(p,x,y)\leq C/2<0$ for $p\in  P$, $x\in \bar U$ and $y\in [0,\delta_1]$. This time, for any  $p\in P$, $x\in \bar U$ and $y\in [0,\delta_1]$ there is an  $\omega(p,x,y)\in (0,1)$ such that
\[
g(p,x,y)= \frac{1}{(n+1)!}\,\frac{\partial^{n+1}g}{\partial y^{n+1}}(p,x,\omega(p,x,y)y)\,y^{n+1}\leq \frac{1}{(n+1)!}\,\frac{C}{2} \,y^{n+1}\,,
\]
and it suffices to take $\wit \rho_0=-C/(2(n+1)!)>0$.  The proof is finished.
\end{proof}
Just note that when $g$  is $n+1$ times continuously differentiable with respect to $y$ and~\eqref{derivadas} holds, if condition (c2) is assumed so that $g(p,x,y)\leq 0$ for $p\in P$, $x\in\bar U$ and $y\geq 0$, then $\frac{\partial^{n+1} g}{\partial y^{n+1}}(p,x,0)\leq 0$ for $p\in P$ and $x\in \bar U$. This means that condition~\eqref{n+1} is not too restrictive in our context or work.
\begin{teor}\label{teor-coro}
Assume conditions $(c1)$-$(c5)$ on $g$ with $r_0=0$ in $(c5)$. Assume first that $g\in C^{0,0,n+\beta}(P\times \bar U\times [0,\delta])$ for an $n\geq 1$, a $\beta\in (0,1^-]$ and a $\delta>0$ and it satisfies~\eqref{derivadas}. Then:
\begin{itemize}
\item[(i)] If $c(t,p_0)^{n-1+\beta}\in L^1((-\infty,0])$ holds for some $p_0\in P$, then  $b(p_0)\gg 0$.
\item[(ii)] If  $\lambda_P>0$, there is an ergodic measure $\mu$ such that $b(p)\gg 0$ for almost every $p$ with respect to $\mu$.
\end{itemize}
Conversely, assume that $g\in C^{0,0,n+1}(P\times \bar U\times [0,\delta])$ for an $n\geq 1$ and a $\delta>0$, and it satisfies~\eqref{derivadas} and~\eqref{n+1};
the principal spectrum $\Sigma_{{\rm pr}}\subset [0,\infty)$; and the boundary conditions are of either Neumann or Robin type. If $b(p_0)\gg 0$ for some $p_0\in P$, then $c(t,p_0)^n\in L^1((-\infty,0])$.
\end{teor}
\begin{proof}
The proof relies on standard comparison methods in monotone skew-product semiflows.
As usual, denote by  $\tau(t,p,z)=(p{\cdot}t,u(t,p,z))$  the  skew-product semiflow induced by~\eqref{pdefamilynl}.
If $g\in C^{0,0,n+\beta}(P\times \bar U\times [0,\delta])$ satisfies~\eqref{derivadas}, Proposition~\ref{prop-analisis 1}.1 implies that there exists a $\rho_0>0$ such that~\eqref{holder} holds. Let us fix an $r>\delta>0$ big enough so that formula~\eqref{b(p)} is applicable for $\tau$ and also for the skew-product semiflow induced by the solutions of~\eqref{chafee-infante} with nonlinear term $-y^{n+\beta}$ for $y\geq 0$. In particular this implies that formula~\eqref{b(p)} also works for the skew-product semiflow $\tau_\theta(t,p,z)=(p{\cdot}t,u_\theta(t,p,z))$ associated with the family~\eqref{chafee-infante} with $\theta=n+\beta>1$, and any $\rho>1$. Now, we fix the appropriate value of $\rho$. It is easy to choose a $\rho\geq \max\{\rho_0,1\}$ so that $g(p,x,y)\geq -\rho\,y^{n+\beta}$ for $p\in P$, $x\in \bar U$ and $y\in [0,r]$.  Then, by Theorem~3.1 in~\cite{calaobsa} (which also applies with Dirichlet boundary conditions), $u(t,p{\cdot}(-t),r e_0)\geq u_\theta(t,p{\cdot}(-t),r e_0)$ for $t\geq 0$, $p\in P$ and $e_0$ the one in~\eqref{b(p)}. Thus, taking limits as $t\to\infty$, $b(p)\geq b_\theta(p)$ for $p\in P$, for $b_\theta$ the upper boundary map of the global attractor for $\tau_\theta$.
Then, to prove (i) just apply Theorem~\ref{teor-atractor no trivial} to ensure that $b_\theta(p_0)\gg 0$, so that also $b(p_0)\gg 0$.
As for (ii), having in mind relations~\eqref{medidas erg}, we can take an ergodic measure $\mu$ such that $\lim_{t\to -\infty} \ln c(t,p)/t=\lambda_P>0$ for almost every $p$ with respect to  $\mu$. Then, for all such $p$, $c(t,p)^{\beta'}\in L^1((-\infty,0])$ for any $\beta'>0$. Taking $\beta'=n-1+\beta$, the result just proved in (i) applies.
\par
Now we focus on the converse result. Proposition~\ref{prop-analisis 1}.2 implies that there exist a $0<\delta_1<\delta$ and a  $\wit \rho_0>0$ such that~\eqref{holder-2} holds. Fix $r>\delta_1>0$ big enough so that formula~\eqref{b(p)} is applicable. This time, it is easy to check that we can take a  $0<\wit \rho\leq \wit \rho_0$ sufficiently small so that $g(p,x,y)\leq -\wit\rho\,y^{n+1}$ for $p\in P$, $x\in \bar U$ and $y\in [0,r]$. Here $\tau_\theta(t,p,z)=(p{\cdot}t,u_\theta(t,p,z))$ is the skew-product semiflow associated with~\eqref{chafee-infante} for this constant $\wit \rho>0$ and $\theta=n+1$. Once more applying Theorem~3.1 in~\cite{calaobsa}, now $u(t,p{\cdot}(-t),r e_0)\leq u_\theta(t,p{\cdot}(-t),r e_0)$ for $t\geq 0$ and $p\in P$. It might be necessary to take a bigger $\wit r>r$ so that formula~\eqref{b(p)} holds for $b_\theta$. In that case, the monotonicity of $\tau_\theta$ is applied to get that $u(t,p{\cdot}(-t),r e_0)\leq u_\theta(t,p{\cdot}(-t),r e_0)\leq u_\theta(t,p{\cdot}(-t),\wit r e_0)$ for $t\geq 0$  and $p\in P$. Taking limits as $t\to\infty$, $b(p)\leq b_\theta(p)$ for any $p\in P$. If $b(p_0)\gg 0$ for some $p_0\in P$, also $b_\theta(p_0)\gg 0$ and  then Theorem~\ref{teor-atractor no trivial} implies that $c(t,p_0)^n\in L^1((-\infty,0])$, as we wanted to prove.
\end{proof}
We remark that a similar result to the one in (ii) can be found in  Theorem~A~2) in Mierczy{\'n}ski and Shen~\cite{mish04} for sublinear parabolic scalar equations. But note that the context above is much more general. Also note that, by Proposition~\ref{prop-puntos asintoticos}, the bigger $\beta>0$ is, the easier it is that $c(t,p_0)^\beta\in L^1((-\infty,0])$.
\par
To finish this section, we state the counterpart of Theorem~\ref{teor-principal spectrum}, collecting  the information we have on the structure of the attractor for a family~\eqref{pdefamilynl} of purely dissipative problems, depending on the  principal spectrum.
\begin{teor}\label{teor-purely diss final}
Assume conditions $(c1)$-$(c5)$ on $g$ with $r_0=0$ in $(c5)$. Then:
\begin{itemize}
\item[(s1)] If $\alpha_P\leq\lambda_P < 0$, then $\A=P\times \{0\}$ and it is uniformly exponentially stable.
\item[(s2)]  If $\alpha_P<0=\lambda_P$, then $P_0=\{p\in P\mid b(p)=0\}$ is a set of complete measure.
\item[(s3)] If  $\alpha_P < 0 <\lambda_P$, then there exists an ergodic measure $\nu$ such that $b(p)=0$ for almost every $p$ with respect to $\nu$;  and, if besides $g\in C^{0,0,1+\beta}(P\times \bar U\times [0,\delta])$ for a $\beta\in (0,1^-]$ and a $\delta>0$, there exists an ergodic measure $\mu$  such that $b(p)\gg 0$ for almost every $p$ with respect to $\mu$. The dynamical description in Theorem~$\ref{teor-sublineal purely-dis}$~{\rm(ii)} applies for all such $p$, provided that also the sublinear condition $(c6)$ holds.
\item[(s4)] If  $\alpha_P = 0 \leq \lambda_P$, then there exists an ergodic measure $\nu$  such that $b(p)=0$ for almost every $p$ with respect to $\nu$. Besides, if $\lambda_P=0$, $P_0$ is a set of complete measure, whereas if $\lambda_P>0$ and $g\in C^{0,0,1+\beta}(P\times \bar U\times [0,\delta])$ for a $\beta\in (0,1^-]$ and a $\delta>0$, there is an ergodic measure $\mu$ as in {\rm(s3)}.
\item[(s5)] If $0<\alpha_P\leq  \lambda_P$, then the semiflow $\tau$ is  uniformly persistent in the interior of the positive cone and $b$ is uniformly strongly positive. If besides $(c6)$ is satisfied, the dynamical description in Theorem~$\ref{teor-sublineal purely-dis}$~{\rm(ii)} applies for all $p$, and $b$ is continuous.
\end{itemize}
\end{teor}
\begin{proof}
(s1), (s2),  the existence of $\nu$ in (s3) and (s4),  and part of (s5) have been proved in the general purely dissipative setting in Theorem~\ref{teor-purely diss 1}. Sublinearity is required in (s5) in order to apply Theorem~\ref{teor-sublineal purely-dis}~(iii).
\par
The existence of $\mu$ in (s3) or (s4) when $\lambda_P>0$ and a $\beta$-H\"{o}lder/Lipschitz continuous variation is assumed on the first partial derivative of $g$ with respect to $y$, has been proved in Theorem~\ref{teor-coro}~(ii): just note that condition~\eqref{derivadas} with $n=1$ reduces to condition (c1).  In this case $\mu(P_+)=1$ for the invariant set $P_+=\{p\in P\mid b(p)\gg 0\}$.
Since   $\mu$ is a regular Borel measure,  we can apply  Lusin's theorem to the semicontinuous, thus measurable function  $\|b\|:P\to \R$, $p\mapsto \|b(p)\|$  to affirm that, fixed an $\varepsilon>0$, there exists a continuous map $\wit b:P\to \R$  such that $\mu(\{p\in P\mid \|b(p)\|=\wit b(p)\})>1-\varepsilon$. Since $\mu$ is regular, we can take  a compact set $E_0\subset P_+\cap \{p\in P\mid \|b(p)\|=\wit b(p)\}$ with $\mu(E_0)>0$. In particular, there exists an $\eta_0> 0$ such that $\inf\{ \|b(p)\|\mid p\in E_0\}\geq \eta_0$.
Finally, an application of Birkhoff's ergodic theorem to the characteristic function of $E_0$ implies that for almost every $p\in P$ with respect to $\mu$ there exists a real sequence $(t_n)_n\uparrow \infty$ such that $p{\cdot}t_n\in E_0$ for every $n\geq 1$. That is, $\|b(p{\cdot}t_n)\|\geq \eta_0$ for $n\geq 1$ and $\limsup_{t\to \infty} \|b(p{\cdot}t)\|>0$,  so that whenever condition (c6) is assumed, Theorem~\ref{teor-sublineal purely-dis}~(ii) applies.
The proof is finished.
\end{proof}
%%%%%%%%%%%%%%%%%%%%%%%%%%%%%%%%%%%%%%%%%%%%%%%%%%%%%%%%%%%%%%%%%%%%%%%%%%%%%%%%%%%%%%%%%%%%%%%%%%%%%%%%%%%%%%%%%
%%%%%%%%%%%%%%%%%%%%%%%%%%%%%%%%%%%%%%%%%%%%%%%%%%%%%%%%%%%%%%%%%%%%%%%%%%%%%%%%%%%%%%%%%%%%%%%%%%%%%%%%%%%%
\section{Some illustrative examples}\label{sec-examples}
Note that the disadvantage of condition \eqref{a1} is that normally we do not know the explicit expression for the 1-dim linear cocycle $c(t,p)$.
In this last section we include some easy examples in the linear-dissipative and the purely dissipative contexts, in which we can calculate everything and draw nice conclusions.
\begin{eje}\label{eje p_0}
Consider the odd, bounded and uniformly continuous map $p_0:\R\to\R$ given by
\[
p_0(t)=\left\{\begin{array}{ll} -\des\frac{2}{t}\,, & \text{if} \;\, t\leq -1\\[0.25cm]
-2\,(t-1)-2\,, & \text{if} \;\, -1\leq t \leq 1\\[0.1cm]
-\des\frac{2}{t}\,, & \text{if} \;\, t\geq 1
\end{array}\right.
\]
which satisfies that the continuous spectrum for $v'=p_0(t)\,v$ is $\Sigma=\{0\}$, and also $e^{\int_0^t \!p_0}\in L^1((-\infty,0])$. Besides, the integral $\int_0^t -p_0=1+2\,\ln t$ grows slower than any $t^\beta$ ($0<\beta<1$) as $t\to\infty$ (this comment comes in relation with Remark~\ref{nota-odes}.3). Also note that the hull of $p_0$ is the compact and connected set $P$ formed by all the translated maps of $p_0$, which we denote as usual by $p_0{\cdot}t$, $t\in \R$  plus the identically null map. Thus, the shift flow on $P$ is not a minimal flow, the only minimal set in it is $\{0\}$ and $\delta_{\{0\}}$, the  Dirac measure at $p=0$, is the only ergodic measure.
\par
(i) {\it A purely dissipative example with $\Sigma_{{\rm pr}}=\{0\}$ and a homoclinic orbit\/}. For $\rho>0$, $\theta>1$ and $a:P\to \R$, $p\mapsto a(p)=p(0)$, let $\tau (t,p,z)=(p{\cdot}t,u(t,p,z))$ be the skew-product semiflow induced by the solutions of the family of purely dissipative problems for $p\in P$, given in the positive cone by
\begin{equation}\label{pdefamily}
\left\{\begin{array}{l} \des\frac{\partial y}{\partial t}  =
 \Delta \, y+(\gamma_0+a(p{\cdot}t))\,y-\rho\,y^\theta\,,\quad t>0\,,\;\,x\in U, \\[.2cm]
By:=\bar\alpha(x)\,y+\kappa\,\des\frac{\partial y}{\partial n} =0\,,\quad  t>0\,,\;\,x\in \partial U,
\end{array}\right.
\end{equation}
where $\gamma_0$ is the first eigenvalue of the BVP~\eqref{bvp} with associated normalized eigenfunction $e_0\gg 0$. It is easy to check that the principal bundle is given by $X_1(p)=\langle e_0\rangle$ for all $p\in P$ and $c(t,p)=e^{\int_0^t a(p{\cdot}s)ds}$ is the 1-dim linear cocycle, so that the principal spectrum is the degenerate interval $\{0\}$, i.e., $\alpha_P=\lambda_P=0$ and Theorem~\ref{teor-purely diss final}~(s4) applies. Thus, there exists an ergodic measure $\nu$ in $P$ such that $b(p)=0$ for almost all $p$ with respect to $\nu$. Since necessarily $\nu=\delta_{\{0\}}$, it follows that $b(0)=0$. By Proposition~\ref{prop-0 o positivo}, $p=0$ is a point of continuity for $b$. Since $\lim_{t\to \pm\infty} p_0{\cdot}t=0$, then  $\lim_{t\to \pm\infty} b(p_0{\cdot}t)=0$.  Thus, looking at $(0,b(0))=(0,0)$ as an equilibrium point for the skew-product semiflow, $\{(p_0{\cdot}t,b(p_0{\cdot}t))\mid t\in\R\}$ is a homoclinic orbit.
\par
Now, for the initial map $p_0\in P$ we can calculate the values $\theta>1$ for which $c(t,p_0)^{\theta-1}\in L^1((-\infty,0])$, in order to guarantee that $b(p_0)\gg 0$ by Theorem~\ref{teor-atractor no trivial}~(i). Some routine calculations show that the previous condition holds if and only if $\theta>3/2$. Therefore, for any $\theta>3/2$, $b(p_0)\gg 0$, and thus also $b(p_0{\cdot}t)\gg 0$ for any $t\in \R$, and the homoclinic orbit is nontrivial. However, for $1<\theta\leq 3/2$, at least with Neumann or Robin boundary conditions, by Theorem~\ref{teor-atractor no trivial}~(ii) we know that $b(p_0)=0$ and the attractor is the trivial one, $\A=P\times\{0\}$.
\par
(ii) {\it A linear-dissipative example with $\Sigma_{{\rm pr}}=\{0\}$ and strongly positive upper boundary map, not bounded away from $0$ and discontinuous\/}. Based on the previous purely dissipative example, we  build a linear-dissipative family of problems for $p\in P$ sharing the linear term, and thus the associated 1-dim cocycle $c(t,p)$,
\begin{equation*}
\left\{\begin{array}{l} \des\frac{\partial y}{\partial t}  =
 \Delta \, y+(\gamma_0+a(p{\cdot}t))\,y+g(y)\,,\quad t>0\,,\;\,x\in U, \\[.2cm]
By:=\bar\alpha(x)\,y+\kappa\,\des\frac{\partial y}{\partial n} =0\,,\quad  t>0\,,\;\,x\in \partial U,
\end{array}\right.
\end{equation*}
for a map $g:\R\to\R$ of class $C^1$ which satisfies (c1)-(c5) with $r_0>0$ in (c5) and such that $ g(y)\geq -\rho\,y^\theta $ for $y\geq 0$, for $\rho>0$ and $\theta>3/2$. Then, denoting by $\wit \tau (t,p,z)=(p{\cdot}t,\wit u(t,p,z))$ the induced skew-product semiflow and by $\wit b(p)$ the upper boundary map of the attractor,  since $\wit u(t,p,z) \geq u(t,p,z)$ for $t\geq 0$, $p\in P$ and  $z\geq 0$,  then $\wit b(p)\geq b(p)$ for $p\in P$ (once more, see~\eqref{b(p)}), so that in particular  $\wit b(p_0{\cdot}t)\gg 0$ for any $t\in \R$. This time also $\wit b(0)\gg 0$, since for $p=0$, $c(t,0)\equiv 1$, and Proposition~\ref{prop-equivalentes} applies. Besides, taking $\beta>0$ such that $\wit b(p_0)\leq \beta\,e_0$ and comparing the solutions of the nonlinear problem and the linear one, $\wit b(p_0{\cdot}t)=\wit u(t,p,\wit b(p_0))\leq \phi(t,p)\,\beta\,e_0=\beta\, c(t,p_0)\,e_0=\beta\, e^{-1-2\ln t}e_0\to 0$ as $t\to \infty$. In all, we have an example in which $\wit b(p)\gg 0$ for any $p\in P$, but $\wit b$ is not uniformly strongly above $0$ and it is discontinuous at $p=0$.
\end{eje}
Next, inspired in an example in Mierczy{\'n}ski and Shen~\cite{mish} (p.~203-204), we modify the map $p_0(t)$ in the previous example in order to obtain a heteroclinic orbit in the global attractor. For convenience, we restrict ourselves to the case of Chafee-Infante equations with  Neumann boundary conditions.
\begin{eje}\label{eje a_0}
Consider the bounded and uniformly continuous map $p_1:\R\to\R$ given by
\[
p_1(t)=\left\{\begin{array}{rl} -\des\frac{2}{t}\,, & \text{if} \;\, t\leq -1\\[0.25cm]
2\,, & \text{if} \;\, t\geq -1\,.
\end{array}\right.
\]
This time, the hull of $p_1$ is the compact and connected set $P$ formed by all the translated maps of $p_1$, which we denote by $p_1{\cdot}t$, $t\in \R$,  plus the identically null map and the identically equal to $2$ map. More precisely, $p_1{\cdot}t\to 0$ as $t\to -\infty$ and $p_1{\cdot}t\to 2$ as $t\to \infty$. Again, the shift flow on $P$ is not minimal and the only minimal sets are $\{0\}$ and $\{2\}$, as well as the only ergodic measures are the Dirac measures  $\delta_{\{0\}}$ and $\delta_{\{2\}}$. Besides, for $a:P \to \R$, $p\mapsto a(p)=p(0)$, the continuous spectrum or Sacker-Sell spectrum of the family of linear equations over the hull, $v'= a(p{\cdot}t)\, v$, $p\in P$ is the interval $[0,2]$.
\par
(i)  {\it A purely dissipative family with $\Sigma_{{\rm pr}}=[0,1]$ and a heteroclinic orbit\/}.
Let us consider the family for $p\in P$,
\begin{equation*}
\left\{\begin{array}{l} \des\frac{\partial y}{\partial t}  =
 \Delta \, y+\frac{1}{2}\,a(p{\cdot}t)\,y-\frac{1}{2}\,y^3\,,\quad t>0\,,\;\,x\in U,
  \\[.2cm]
\des\frac{\partial y}{\partial n} =0\,,\quad  t>0\,,\;\,x\in \partial U,
\end{array}\right.
\end{equation*}
and denote by $\tau$ the induced skew-product semiflow. Note that in the Neumann case $\gamma_0=0$ and the associated 1-dim cocycle is $c(t,p)=e^{\int_0^t \frac{1}{2}\,a(p{\cdot}s)ds}$ with principal spectrum $\Sigma_{{\rm pr}}=[0,1]$. Besides, for $p=0$ and $p=2$ we have two equilibrium points, namely $(0,0)$ and $(2,\sqrt{2})$, for the skew-product semiflow. In particular this implies that  $b(2)\gg 0$. In what respects to $p_1$, since it coincides with the map $p_0$ in Example~\ref{eje p_0} for times $t\leq -1$, $e^{\int_0^t \!p_1}\in L^1((-\infty,0])$. Then, as $c(t,p_1)^2=e^{\int_0^t \!p_1}$, Theorem~\ref{teor-atractor no trivial}~(i) with $\theta=3$ implies that $b(p_1)\gg 0$, so that $b(p_1{\cdot}t)\gg 0$ for $t\in\R$.  Besides, Theorem~\ref{teor-purely diss final}~(s4) applies and, with the former description, it must be $\nu=\delta_{\{0\}}$ and $b(0)=0$. In particular $p=0$ is a continuity point for $b$ and $\lim_{t\to -\infty} (p_1{\cdot}t,b(p_1{\cdot}t))=(0,0)$. Finally,  $\lim_{t\to \infty} (p_1{\cdot}t,b(p_1{\cdot}t))=(2,\sqrt{2})$, so that the orbit $\{(p_1{\cdot}t,b(p_1{\cdot}t))\mid t\in\R\}$ connects both equilibria. To see this, we turn to the explicit expression of $b(p_1{\cdot}t)$ which can be extracted from the proof of Lemma~\ref{lema-odes} and we do some routine calculations.
\par
(ii) {\it A linear-dissipative example with $\Sigma_{{\rm pr}}=[0,1]$ and strongly positive upper boundary map\/}. As in Example \ref{eje p_0} (ii), we now build a linear-dissipative family for $p\in P$, majorating the previous purely dissipative family in the positive cone and sharing its linear part, so that in particular $\Sigma_{\rm{pr}}=[0,1]$. Then, if $\wit b(p)$ denotes the upper boundary map of the attractor, $\wit b(p)\gg 0$ for any $p\in P$: the only doubt might be for $\wit b(0)$, but the cocycle for $p=0$ is  $c(t,0)\equiv 1$, and Proposition~\ref{prop-equivalentes} applies in the linear-dissipative setting.
\end{eje}
To finish, we present a final example to show that in the purely dissipative case, with principal spectrum $\Sigma_{{\rm pr}}=[\alpha_P,0]$ with $\alpha_P<0$, the attractor might be the trivial one $\A=P\times \{0\}$: see Theorem~\ref{teor-purely diss final}~(s2) and compare with Theorem~\ref{teor-principal spectrum}~(s2), saying that this cannot happen in the linear-dissipative context.
\begin{eje}\label{eje p_2}
 {\it A purely dissipative family with $\Sigma_{{\rm pr}}=[-1,0]$ and a trivial attractor\/}.
Consider the bounded and uniformly continuous map $p_2:\R\to\R$ given by
\[
p_2(t)=\left\{\begin{array}{rl} \des\frac{1}{t}\,, & \text{if} \;\, t\leq -1\\[0.25cm]
-1\,, & \text{if} \;\, t\geq -1\,.
\end{array}\right.
\]
This time, the hull of $p_2$ is the compact and connected set $P$ formed by all the translated maps of $p_2$,  $\{p_2{\cdot}t\mid t\in \R\}$,  plus the identically null map and the identically equal to $-1$ map. Moreover, $p_2{\cdot}t\to 0$ as $t\to -\infty$ and $p_2{\cdot}t\to -1$ as $t\to \infty$. Again, the shift flow on $P$ is not minimal and the only minimal sets are $\{0\}$ and $\{-1\}$, as well as the only ergodic measures are the Dirac measures  $\delta_{\{0\}}$ and $\delta_{\{-1\}}$. Besides, for $a:P \to \R$, $p\mapsto a(p)=p(0)$, the Sacker-Sell spectrum of the family of linear equations over the hull, $v'= a(p{\cdot}t)\, v$, $p\in P$ is the interval $[-1,0]$.
\par
For $\rho>0$ and $\theta>1$, let $\tau$ be the skew-product semiflow induced by the solutions of the family of purely dissipative problems for $p\in P$, given in the positive cone by \eqref{pdefamily},
where $\gamma_0$ is the first eigenvalue of the BVP~\eqref{bvp}. Recall that then $c(t,p)=e^{\int_0^t a(p{\cdot}s)ds}$ is the associated 1-dim linear cocycle, so that the principal spectrum is
$\Sigma_{{\rm pr}}=[-1,0]$. Let us see that $b(p)=0$ for all $p\in P$. First of all note that for $p=0$, the restricted family over the minimal set $M=\{0\}$ which is uniquely ergodic, has trivial bounded cocycle $c(t,0)=1$, $t\in\R$ and null Lyapunov exponent. Thus, by Proposition~5.12 in~\cite{calaobsa} we know that $b(0)=0$.
\par
Finally, it is straightforward that both  $c(t,-1)$ and $c(t,p_2)$ tend to $\infty$ as $t\to -\infty$. Then, by Remark \ref{nota-1 purely},
$b(-1)= b(p_2)=0$, and then also  $b(p_2{\cdot}t)=0$ for $t\in \R$ and we are done.
\end{eje}
%%%%%%%%%%%%%%%%%%%%%%%%%%%%%%%%%%%%%%%%%%%%%%%%%%%%%%%%%%%%%%%%%%%%%%%%%%%%%%%%%%%%%%%


\begin{thebibliography}{99}
%\bibitem{brcv} {\sc R.C.D.S. Broche, A.N. Carvalho, J. Valero}, A non-autonomous scalar one-dimensional dissipative parabolic problem: the description of the dynamics, {\em Nonlinearity} \textbf{32} (12) (2019), 274--299.
\bibitem{caloNonl} {\sc T. Caraballo,  J.A. Langa, R. Obaya}, Pullback, forward and chaotic dynamics in 1-D non-autonomous linear-dissipative equations, {\em Nonlinearity} \textbf{30} (1) (2017), 274--299.
\bibitem{calaobsa} {\sc T. Caraballo,  J.A. Langa, R. Obaya, A.M. Sanz}, Global and cocycle attractors for non-autonomous reaction-diffusion equations. The case of null upper Lyapunov exponent, {\em J. Differential Equations} \textbf{265}  (2018), 3914--3951.
\bibitem{cardoso} {\sc C.A. Cardoso, J.A. Langa, R. Obaya}, Characterization of cocycle attractors for nonautonomous reaction-diffusion equations,   {\em Internat. J. Bifur. Chaos\/}, \textbf{26} (8)  id.1650135-263 (2016).
\bibitem{calaro12} {\sc A.N. Carvalho, J.A. Langa, J.C. Robinson}, Structure and bifurcation of pullback attractors in a non-autonomous Chafee-Infante equation, {\em Proc. Amer. Math. Soc.} \textbf{140} (7) (2012),  2357--2373.
\bibitem{calaro} {\sc A.N. Carvalho, J.A. Langa, J.C. Robinson}, {\em Attractors for Infinite-Dimensional Non-Autonomous Dynamical Systems\/}, Applied Mathematical Sciences, Vol. {\bf 182}, Springer, New York, 2013.
\bibitem{chin} {\sc N. Chafee, E.F. Infante}, A bifurcation problem for a nonlinear partial differential equation of parabolic type,  {\em Applicable Anal.} {\bf 4} (1974), 17--37.
\bibitem{chue} {\sc I. Chueshov}, \textit{Monotone Random
        Systems. Theory and Applications}, Lecture Notes in Math. \textbf{1779},
        Springer-Verlag, Berlin, Heidelberg, 2002.
\bibitem{comi}  {\sc C.C. Conley, R.K. Miller}, Asymptotic stability without uniform stability: almost periodic coefficients, {\em J. Differential Equations} {\bf 1} (1965), 333--336.
\bibitem{copp} {\sc W.A. Coppel},  \textit{Dichotomies in Stability Theory}, Lecture Notes in Math. \textbf{629}, Springer-Verlag, New York, 1978.
\bibitem{elli} {\sc R. Ellis},  \textit{Lectures on Topological Dynamics\/},
        Benjamin, New York, 1969.
\bibitem{fita} {\sc P.C. Fife, M.M. Tang},
       Comparison Principles for Reaction-Diffusion Systems: irregular comparison functions
       and applications to questions of stability and speed of propagation of disturbances,
       {\em J. Differential Equations\/} {\bf 40} (1981), 168--185.
\bibitem{frie} {\sc A. Friedman}, {\em Partial Differential Equations of Parabolic Type},
       Prentice-Hall, Englewood Cliffs, N.J., 1964.
\bibitem{joklpa} {\sc R. Johnson, P. Kloeden,  R. Pavani}, Two-step transition in nonautonomous bifurcations: an explanation,  {\em Stoch. Dyn.} \textbf{2} (1) (2002), 67--92.
\bibitem{jomo} {\sc R. Johnson, J. Moser}, The rotation number for almost periodic potentials,
        {\em Commun. Math. Phys.} {\bf 84} (3) (1982), 403--438.
\bibitem{jops} {\sc R. Johnson, K.J. Palmer, G.R. Sell}, Ergodic properties of linear dynamical systems, {\em SIAM J. Math. Anal.} \textbf{18} (1) (1987), 1--33.
\bibitem{klra} {\sc P.E. Kloeden, M. Rasmussen}, {\em Nonautonomous Dynamical Systems\/}, AMS Mathematical Surveys and Monographs, Vol. {\bf 176}, AMS, Providence, 2011.
\bibitem{laos} {\sc J.A. Langa, R. Obaya, A.M. Sanz}, Forwards attraction properties in scalar non-autonomous  linear-dissipative parabolic PDEs. The case of null upper Lyapunov exponent, {\em Nonlinearity} \textbf{33} (9) (2020), 4277--4309.
\bibitem{mish} {\sc J. Mierczy{\'n}ski, W. Shen}, Exponential separation and principal Lyapunov exponent/spectrum for random/nonautonomous parabolic equations, {\em J. Differential Equations} \textbf{191}  (2003), 175--205.
\bibitem{mish04} {\sc J. Mierczy{\'n}ski, W. Shen}, Lyapunov
        exponents and asymptotic dynamics in random Kolmogorov models,
        {\em J. Evol. Equ.} {\bf 4} (2004), 371--390.
\bibitem{mish08} {\sc J. Mierczy{\'n}ski, W. Shen}, {\em Spectral Theory for Random and Nonautonomous Parabolic Equations and Applications}, Monographs and Surveys in Pure and Applied Mathematics, Vol. {\bf 139}, Chapman and Hall/CRC, 2008.
\bibitem{nono2} {\sc S. Novo, C. N\'{u}\~{n}ez, R. Obaya}, Almost automorphic and
        almost periodic dynamics for quasimonotone
        non-autonomous functional differential equations,
        {\em J. Dynamics Differential Equations} {\bf 17} (3)
        (2005), 589--619.
\bibitem{noos7} {\sc S. Novo, R. Obaya, A.M. Sanz}, Uniform persistence and upper Lyapunov exponents for monotone skew-product semiflows, {\em Nonlinearity} \textbf{26} (2013), 2409--2440.
\bibitem{nuob} {\sc C. N\'{u}\~{n}ez, R. Obaya}, Li-Yorke chaos in nonautonomous Hopf bifurcation patterns-I, {\em Nonlinearity} \textbf{32}  (2019), 3940–-3980.
\bibitem{nuos} {\sc C. N\'{u}\~{n}ez, R. Obaya, A.M. Sanz},
        Minimal sets in monotone and sublinear skew-product semiflows I:
        The general case, {\it J. Differential Equations\/} \textbf{248} (2010), 1879--1897.
\bibitem{obsa2019} {\sc R. Obaya, A.M. Sanz}, Persistence in non-autonomous quasimonotone parabolic partial functional differential equations with delay, {\em Discrete Contin. Dyn. Syst. Series B} \textbf{24} (8) (2019), 3947--3970.
\bibitem{poin} {\sc H. Poincar\'{e}},  Sur les s\'{e}ries trigonom\'{e}triques, {\it C.R. Acad. Sci.} {\bf 101} (1885), 1131--1134.
\bibitem{pote} {\sc P. Pol\'{a}\v{c}ik, I. Tere\v{s}\v{c}\'{a}k}, Exponential
        separation and invariant bundles for maps in ordered Banach spaces
         with applications to parabolic equations, {\em J. Dynamics
        Differential Equations} {\bf 5} No. 2 (1993), 279--303.
\bibitem{sase76} {\sc R.J. Sacker, G.R. Sell}, Existence of dichotomies and invariant splittings for linear differential systems, II, {\em J. Differential Equations} {\bf 22} (1976), 478--496.
\bibitem{sase} {\sc R.J. Sacker, G.R. Sell}, A spectral theory for
        linear differential systems, {\em J. Differential Equations}
               {\bf 27} (1978), 320--358.
\bibitem{sase94} {\sc R.J. Sacker, G.R. Sell}, Dichotomies for
        linear evolutionary equations in Banach spaces, {\em J.
        Differential Equations} \textbf{113} (1994), 17--67.
\bibitem{selg} {\sc J. Selgrade}, Isolated invariant sets for flows on vector bundles,
         {\em Trans. Amer. Math. Soc.} {\bf 203} (1975), 359--390.
\bibitem{shyi} {\sc W. Shen, Y. Yi}, {\em Almost Automorphic and Almost
        Periodic Dynamics in Skew-Product Semiflows},  Mem. Amer. Math. Soc.
        {\bf 647}, AMS, Providence,  1998.
\bibitem{shyi2} {\sc W. Shen, Y. Yi}, Convergence in almost periodic Fisher and Kolmogorov models, {\em J. Math. Biol.} {\bf 37} (1998), 84--102.
\bibitem{shne} {\sc Ya. Shneiberg}, Zeros of integrals along trajectories of
        ergodic systems, {\em Funktsional. Anal. i Prilozhen.} {\bf 19}
        (2) (1985), 92--93.
\bibitem{zhle} {\sc V.V. Zhikov, B.M. Levitan}, Favard theory,
        {\em Russian Math. Surveys} {\bf 32} (1977), 129--180.
\end{thebibliography}
\end{document}